\documentclass{amsart}

\usepackage[utf8]{inputenc} 
\usepackage[T1]{fontenc}    
\usepackage{hyperref}       
\usepackage{url}            
\usepackage{booktabs}       
\usepackage{amsfonts}       
\usepackage{nicefrac}       
\usepackage{microtype}      
\usepackage{lipsum}
\usepackage{lmodern}

\usepackage[all]{hypcap}                

\usepackage{amsmath}
\usepackage{amssymb}
\usepackage{amsthm}
\usepackage{geometry}
\usepackage{enumerate}
\usepackage{tikz-cd}
\usepackage{tikz}
\usetikzlibrary{positioning}
\usepackage{mathtools}
\usepackage{multicol}
\usepackage{tabularx} 
\usepackage{mathtools}

\newtheorem{theo}{Theorem}[section]
\newtheorem{prop}[theo]{Proposition}
\newtheorem{defi}[theo]{Definition}
\newtheorem{notat}[theo]{Notation}
\newtheorem{lem}[theo]{Lemma}
\newtheorem{cor}{Corollary}[theo]

\newtheorem{rem}[theo]{Remark}
\newtheorem{ex}[theo]{Example}

\title{A study of  $n$-valued non-split maps and applications}

	\author{Daciberg Lima Gon\c calves} 
	\address[Daciberg Lima Gon\c calves]{Departamento de Matem\'atica -- IME -- USP, Rua do Mat\~ao, 1010,  Cidade Universit\'aria, CEP:  05508-090, 
		S\~ao Paulo - SP, Brazil}
	\email{dlgoncal@ime.usp.br}
	
	\author{Bartira Mau\'es}
	\address[Bartira Mau\'es]{ }
\email{bartira.ira@gmail.com }

\author{Daniel Vendr\'uscolo}
\address[D. Vendr\'uscolo]{Universidade Federal de S\~ao Carlos\\
Departamento de Matem\'atica\\
S\~ao Carlos - SP, Brazil} 
\email{daniel.vendruscolo@ufscar.br}

\begin{document}
	\maketitle

	
	\begin{abstract}
		For a given $n$-valued map, there exists a suitable minimal covering
		of the domain such that the composition of this covering with the map
		yields a split map, composed of $n$ single-valued maps, called the lift factors.
		Using the group action of the deck-transformation group we construct intermediate covering spaces as well as intermediate configuration spaces, allowing the original map to lift to these spaces, building a tower to illustrate at which level the lift factors split off.	 
		We then display two applications of these new tools.
		The first is to	derive results concerning the Borsuk-Ulam property for certain
		coverings of the tower.
		The second is to obtain a sharp
		formula for the Nielsen number of an $n$-valued non-split map, using the
		coincidence number of specific single-valued maps that arise in the
		lifts of the tower.	
	\end{abstract}


\section{Introduction}
	The exploration of Nielsen theory for $n$-valued maps on a space $X$ begins with the foundational work of Helga Schirmer \cite{schirmer1984fix, schirmer1984index, schirmer1985minimum}.
	This line of research is part of the study of $n$-valued maps between two spaces, where Nielsen theory plays a central role.
	Many classical questions regarding Nielsen theory for single-valued maps can be reformulated in the context of $n$-valued maps.
	We highlight several works relevant to the present paper:
	\cite{brown2018topology},   \cite{brown2023lift}, \cite{gonccalves2017fixed},	\cite{gonccalves2018fixed},  \cite{gonccalves2024borsuk}, \cite{gonccalves2022free} and	\cite{staecker2021partitions}.
	In \cite{brown2018topology}, the study of $n$-valued maps is reduced to that of	single-valued maps into a suitable configuration space.
	In \cite{gonccalves2017fixed} and \cite{gonccalves2018fixed}, the Nielsen theory of $n$-valued maps on surfaces is investigated, including techniques for computing Nielsen numbers.
	In \cite{staecker2021partitions}, $n$-valued maps are decomposed into unions of (possibly non-single-valued) maps as a tool to study Nielsen theory for $n$-valued maps.
	In \cite{brown2023lift}, a method is developed to study $n$-valued maps using certain associated single-valued maps, which are called the lift factors.
	References \cite{gonccalves2024borsuk} and \cite{gonccalves2022free} explore connections between $n$-valued maps and the Borsuk-Ulam property.

	The Borsuk-Ulam property for a triple $(\tilde X, \tau; Y)$, where $\tilde X$ and $Y$ are topological spaces and $\tau$ is a free $\mathbb{Z}_2$-action on $\tilde X$, is a notion inspired by the classical Borsuk-Ulam theorem.
	Namely {\it the triple $(\tilde X, \tau; Y)$ satisfies the  Borsuk-Ulam property, if for every map $f:\tilde X\to Y$, there is a point  $x\in \tilde  X$, such  that $f(x)=f(\tau(x))$.
	It is possible to generalize this concept to a group action rather than an involution (see \cite{gonccalves2022free}), as well as to a homotopy class in $[\tilde X,Y]$.}
	For that purpose, let $\tilde X$ and $Y$ be topological spaces and let $G$ be a non-trivial group with free action $G\times \tilde X\rightarrow\tilde X$. 
	Then every $g\in G$ induces a homeomorphism on $\tilde X$, which we denote, abusing notation, as $g:\tilde X\rightarrow \tilde X$.
	We say that a \textit{homotopy class $[f_0]\in [\tilde  X, Y]$ has the Borsuk-Ulam property with respect to $G$} if, for every $f\in [f_0]$, there exist distinct $g_1, g_2\in G$ and an $\tilde x\in \tilde X$ such that $f(g_1(\tilde x))=f(g_2(\tilde x))$.

	The first main goal of the present work is to refine the technique developed in \cite{brown2023lift} in order to explore further questions concerning $n$-valued maps.
	More specifically, to a given $n$-valued map $\phi: X\multimap Y$, we associate a sequence of maps beginning with its correspondent map $\Phi:X\rightarrow D_n(Y)$ and ending with the so-called split map $\hat{\Phi}: \tilde{X} \to F_n(Y)$, which we illustrate in a tower diagram.
	These intermediate maps form a new class of objects that can be used to study properties and questions involving $n$-valued maps.
	The second part of the work consists of applications of the developed framework.
	One application explores the relation between the Borsuk-Ulam property and several of the spaces involved in the tower construction, including the original space $X$.
	Another application provides a formula for computing the Nielsen number of an $n$-valued map on any closed manifold (whether orientable or not).
	We hope that the theory developed here will contribute to the growing interest in the study of $n$-valued maps on surfaces and their connections to the	Borsuk-Ulam property.

	This paper contains four sections, in addition to the introduction:
	Section \ref{sec:prelim} recalls several facts about configuration spaces and $n$-valued maps, fixes notation, and proves Proposition \ref{fsigmak=fdeltak}, a key result about lift factors in respect to the group action by the deck-transformation group, which is instrumental in our analysis.\\
	
	\noindent
	{\bf Proposition \ref{fsigmak=fdeltak}}
	\textit{Let $\phi:X\multimap Y$ be an $n$-valued non-split map with correspondent map $\Phi:X\rightarrow D_n(Y)$.
		Let $\hat\Phi=(f_1,\ldots,f_n):\tilde X\rightarrow F_n(Y)$ be the lift of $q\circ \Phi$, as in Notation \ref{notation non-split}.
		Then, for every deck transformation $\delta\in \mathcal Deck(\tilde X,q)$ with associated permutation $\tau=\Gamma(\delta)$ and for every $i\in\{1,\ldots,n\}$ and $k\in \mathbb N$, it holds that 
		$f_{\tau^{k}(i)}=f_i\circ\delta^{k}$.}\\
	
	Section \ref{sec:lift factors} extends the approach introduced in \cite{brown2023lift},
	studying the interaction of the group action $L'\simeq \mathcal Deck(\tilde X,q)$ on the lift factors and how certain group properties influence the lift factors in respect to each other.
	Immediately after Proposition \ref{int cov L Psi=> G<Li} we construct a tower of intermediate configuration spaces, as well as a tower of finite coverings of the domain of the given $n$-valued function.
	Furthermore we define maps (see Notation \ref{notation maps}) between the towers with the desired commutative properties.
	Several examples are provided to illustrate these results.
 
	Section \ref{sec:BU} describes the connection between the new objects introduced in Section \ref{sec:lift factors} and the Borsuk-Ulam property, arising in some applications to this theory.
	To point out one of the results consider a free action of $\mathbb{Z}_2\oplus\mathbb{Z}_2$ on $\mathbb T$ such that the orbit space is homeomorphic to $\mathbb T$, and let $\mathbb{D}$ be the $2$-dimensional disk.\\
	
	\noindent
	{\bf Corollary \ref{Z2+Z2 nao BU}}
	\textit{The triple $({\mathbb T}, \mathbb{Z}_2 \oplus \mathbb{Z}_2, {\mathbb D})$ does not have the Borsuk–Ulam property.}\\                                                                                                                                                                                                                                                                                                                                                                                                                                               
	
	Other main results of Section~\ref{sec:BU} are Propositions~\ref{NonBUmult} and Theorem~\ref{transitive, semireg=> phi}.
	Proposition~\ref{NonBUmult} is a reformulation of a result given in \cite{gonccalves2024borsuk} and \cite{gonccalves2022free}, showing how to construct an $n$-valued map starting from a homotopy class $[f_0]\in [\tilde X,Y]$ which does not satisfy the Borsuk-Ulam property in respect to the free group action $G$ acting on $\tilde X$.
	Before stating Theorem \ref{transitive, semireg=> phi} and its Corollary, which provide a partial converse of Proposition~\ref{NonBUmult}, recall that $L' \subset \mathcal{S}_n$ acts transitively on $\{1,\ldots,n\}$, if for each $i,j\in \{1,\ldots,n\}$ there exists an $\tau \in L'$ such that $\tau(i)=j$.\\
	
	\noindent
	{\bf Theorem \ref{transitive, semireg=> phi}}
	\textit{Let $\phi:X\multimap Y$ be an $n$-valued non-split map with correspondent map $\Phi:X\rightarrow D_n(Y)$ and let $\hat \Phi=(f_1,\ldots,f_n):\tilde X\rightarrow F_n(Y)$, $L'$ and $\Gamma:\mathcal Deck(\tilde X,q)\stackrel{\sim}{\rightarrow}L'$ be as in Notation \ref{notation non-split}.
		Let furthermore $H'\subseteq L'$ be a subgroup,	then the following holds;
		\begin{enumerate}[a)]
			\item
			if $L'$ acts transitively on $\{1,\ldots,n\}$, then there exists, for every $j\in \{1,\ldots,n\}$ a deck transformation $\delta_j\in \mathcal Deck(\tilde X,q)$ such that $f_j=f_1\circ \delta_j$;
			\item
			if, for a subgroup $H'\subseteq L'$ the stabilizer $H'_i$ of an $i\in \{1,\ldots,n\}$ in $H'$ is trivial, then $[f_i]$ does not satisfy the Borsuk-Ulam property in respect to $H=\Gamma^{-1}(H')$.\\
 	\end{enumerate}}
	
	\noindent
	{\bf Corollary \ref{reg=>phi}}
	\textit{Let	$\phi:X\multimap Y$ be an $n$-valued non-split map as in Theorem \ref{transitive, semireg=> phi}
		and let $L'$ act regularly on $F_n(Y)$.
		Then it holds for every $i\in \{1,\ldots,n\}$ that there is a $\delta_i\in \mathcal Deck(\tilde X,q)$ such that $f_i=f_1\circ \delta_i$ and the homotopy class $[f_1]$ does not satisfy the Borsuk-Ulam property in respect to $\mathcal Deck(\tilde X, q)$.}\\

	Section \ref{sec:Nielsen nr of nval non-split} addresses the problem of computing the Nielsen number of an $n$-valued map $\phi$ when $X$ is a closed manifold (either orientable or non-orientable).
	Proposition \ref{index fix=coin} establishes a relation between the Nielsen fixed point theory of $\phi$ and the coincidence theory of certain pairs of maps.
	This leads to the Theorem \ref{Nielsennumber non-orientable}, where we compute the Nielsen number for an $n$-valued non-split map $\phi:X\multimap X$, where $X$ is an arbitrary compact manifold without boundary (orientable or non-orientable), in terms of the Nielsen coincidence number of certain pairs of single valued maps which appear in our tower.\\
	
		\noindent
	{\bf Theorem \ref{Nielsennumber non-orientable}}
	\textit{Let $X$ be a compact manifold without boundary (orientable or non-orientable), and let $\phi:X\multimap X$ be an $n$-valued non-split map with  $\Phi:X\rightarrow D_n(X)$, $q:\tilde X\rightarrow X$, $\hat{\Phi}=\{f_1,\ldots,f_n\}:\tilde{X}\rightarrow F_n(X)$ and $L'$ as in Notation \ref{notation non-split}. Let furthermore $\mathbb I_0$, $L_{i}'$, $ L_i$ and $q_i, \bar f_i:\tilde X/L_i\rightarrow X$ be as in Notation \ref{notation action Li}.
		Then the Nielsen number of $\phi$ is equal to
		\begin{equation*}
			N(\phi) = \sum_{i\in  \mathbb I_0}N(q_i,\bar f_{i}).
	\end{equation*}}

	{\bf Acknowledgement:}  This work was initiated while the second author was conducting a postdoctoral fellowship at the Universidade Federal de S\~ao Carlos. This study was financed, in part, by the São Paulo Research Foundation (FAPESP), Brasil. Process Number 2022/16455-6.
	

\section{Preliminaries}\label{sec:prelim}
	In this section we recall some known facts about $n$-valued maps, establishing the terminology and laying the foundation for our approach to study $n$-valued maps.
	We assume that every topological space is a path-connected, locally path-connected and semilocally simply-connected Hausdorff space.
	Let $X$ and $Y$ be such spaces, then an \textit{$n$-valued map} from $X$ to $Y$, denoted by $\phi:X \multimap Y$ or $\phi:X \stackrel{n}\multimap Y$, is a multimap, specifically a (continuous) correspondence $\phi$ that assigns for each $x \in X$ a subset of $n$ points $\{y_1, \dots, y_n\} \subseteq Y$. 
	
	We  follow the path conducted by D. L. Gon\c calves and J. Guaschi in the works \cite{gonccalves2017fixed} and \cite{gonccalves2018fixed}, where a methodology is formulated to explore $n$-valued maps on closed surfaces utilizing braid theory.
	This approach uses configuration spaces (ordered and unordered) to study $n$-valued maps.
	The \textit{ordered configuration space} of $Y$, denoted by $F_n(Y)$, is formed by taking the $n$-th Cartesian product of $Y$, excluding the fat diagonal given by $\Delta = \{(y_1, \ldots, y_n) \in Y^n : \exists i, j \in \{1, \ldots, n\}, i \neq j \text{ such that } y_i = y_j\}$.
	Therefore $F_n(Y) = Y^n \setminus \Delta = \{(y_1, \ldots, y_n) : y_i \neq y_j \text{ for } i \neq j\}$.
	Consider the action of the $n$-th symmetric group $\mathcal{S}_n$ on $F_n(Y)$, defined by $(\tau, (y_1, \ldots, y_n)) \mapsto \tau\cdot (y_1, \ldots, y_n)=(y_{\tau(1)}, \ldots, y_{\tau(n)})$.
	By taking the quotient of $F_n(Y)$ under this action by $\mathcal{S}_n$, we obtain the \textit{unordered configuration space} $D_n(Y) = F_n(Y)/\mathcal{S}_n$. 
	Note that there exists a natural bijection between the set of $n$-point subsets of a topological space $Y$ and the unordered configuration space $D_n(Y)$.
	This yields a bijective correspondence between $n$-valued maps from $X$ to $Y$ and maps from $X$ to $D_n(Y)$.
	From this point forward, we denote by $\Phi$ the map from $X$ to $D_n(Y)$ that corresponds to the $n$-valued map $\phi: X \multimap Y$. 
	This correspondence preserves continuity under relatively mild conditions imposed on $X$ (see \cite[Corollary 4.1]{brown2018topology}), and by extension, homotopy as well.
	The ordered and unordered configuration space, $F_n(Y)$ and $D_n(Y)$ respectively, are the more common configuration spaces, in this paper we also use  intermediate covering spaces of $F_n(Y)$ and $D_n(Y)$.
	To recall the concept of an intermediate covering space let $q:\tilde X\rightarrow X$ be a covering space, with $X=\tilde X/\mathcal Deck(\tilde X,q)$, where $\mathcal Deck(\tilde X,q)$ is the group of deck-transformations on $\tilde X$.
	An \textit{intermediate covering space} is a space $Z$ together with covering space maps $q_1:\tilde X\rightarrow Z$ and $q_2: Z\rightarrow X$ such that $q_2\circ q_1=q$.
	In this case there exists a subgroup $H$ of $\mathcal Deck(\tilde X,q)$, such that $Z$ is homeomorphic to $\tilde X/H$ and we obtain the composition of covering spaces $\tilde X\rightarrow \tilde X/H\rightarrow X$ (\cite[page 81]{hatcher2002algebraic}).
	An intermediate covering space $Z=\tilde X/H$ is called \textit{proper} if $Z$ is neither $ \tilde X$ nor $X$, i.e. if $H$ is a proper subgroup of $\mathcal Deck(\tilde X,q)$.
		
	We say that an $n$-valued map $\phi:X \multimap Y$ \textit{splits} if there exist (continuous) functions $f_1, \dots, f_n: X \rightarrow Y$ such that $\phi(x) = \{f_1(x), \ldots, f_n(x)\}$ holds for all $x \in X$, and we denote this as \textit{$\phi = \{f_1, \ldots, f_n\}$}.
	Thus, for a split map, there is a continuous manner to associate an $n$-tuple with every point in the image.
	If a map is not split, we call it a \textit{non-split map}.
	According to \cite[page 8]{gonccalves2018fixed}, an $n$-valued map $\phi:X \multimap Y$ splits, if and only if, the function $\Phi$ possesses a \textit{lift}, meaning that there exists a map $\hat{\Phi}=(f_1,\ldots,f_n): X \rightarrow F_n(Y)$ such that $\Phi=\pi \circ \hat \Phi$, where $\pi:F_n(Y)\rightarrow D_n(Y)$ is the covering map.
	The single-valued maps $f_1,\ldots, f_n$ are called the \textit{lift factors of $\Phi$}.
	In this case we have the following commutative diagram and its  correspondent, induced in $\pi_1$, where we use that  $\pi_1(D_n(Y))$ is isomorphic to the braid group $B_n(Y)$ and $\pi_1(F_n(Y))$ to the pure braid group $P_n(Y)$; 
	
	\begin{multicols}{2}
		\noindent
		\begin{equation*}
			\begin{tikzcd}[column sep=huge]
				&F_n(Y)\arrow{d}{\pi}\\
				X\arrow[dashed]{ur}{\hat \Phi}\arrow{r}[swap]{\Phi} &D_n(Y).
			\end{tikzcd}
		\end{equation*}
		
		\noindent
		\begin{equation*}
			\begin{tikzcd}[column sep=huge]
				&P_n(Y)\arrow{d}{\pi_\#}\\
				\pi_1(X)\arrow[dashed]{ur}{\hat\Phi_\#}\arrow{r}[swap]{\Phi_\#} &B_n(Y).
			\end{tikzcd}
		\end{equation*}
	\end{multicols}
	
	For an $n$-valued non-split map $\phi:X\multimap Y$ we follow the approach of R. Brown and D. Gon\c calves in \cite{brown2023lift} to find correspondent lift factors for a non-split map.
	Let $\phi:X\multimap Y$ be an $n$-valued non-split map and $\Phi:X\rightarrow D_n(Y)$ its correspondent map.
	Note that the map induced by $\Phi$ in $\pi_1$, i.e. $\Phi_\#:\pi_1(X)\rightarrow \pi_1(D_n(Y))$ is a map from $\pi_1(X)$ to the $n$-th braid group $B_n(Y)$.	
	Given an arbitrary braid $\beta$, tracing along the strands, one finds that the point $\beta(1)$ is permuted relative to the point $\beta(0)$.
	This permutation is the called the underlying permutation of $\beta$.
	Let $\rho:B_n(X)\rightarrow \mathcal S_n$ take each braid to its underlying permutation, i.e. $\rho(\beta)=\tau$, where $\beta(0)=(p_1,\ldots,p_n)$ and $\beta(1)=(p_{\tau(1)},\ldots,p_{\tau(n)})$.
	Let $\theta=\rho\circ \Phi_\#:\pi_1(X)\rightarrow \mathcal S_n$ be their composition
	, then we obtain the following commutative diagram;
	\begin{equation*}
		\begin{tikzcd}[column sep=huge]
			\pi_1(X) \arrow{r}{\Phi_\#}\arrow{dr}[swap]{\theta}&B_n(Y)\arrow{d}{\rho}\\
			&\mathcal S_n.
		\end{tikzcd}
	\end{equation*}
	Note that $ker\rho$ is equal to $\pi_\#(P_n(Y))$, the subgroup of the pure braids in $B_n(Y)$, consequently $ker\theta=\Phi_\#^{-1}(\pi_\#(P_n(Y)))$, which we denote, abusing notation, by $\Phi_\#^{-1}(P_n(Y))$.
	Since $\phi$ is a non-split map, there is a loop in $\pi_1(X)$, which is not contained in $\Phi_\#^{-1}(P_n(Y))$.
	Otherwise, by the lifting criterion (\cite[p.61, Proposition 1.33.]{hatcher2002algebraic}), there would exist a lift $\hat \Phi:X\rightarrow F_n(X)$ and $\phi$ would be split.
	Therefore $ker\theta$ is a proper subset of $\pi_1(X)$ of finite index, hence by \cite[p. 66, Proposition 1.36]{hatcher2002algebraic} there is a finite covering space $q:\tilde X\rightarrow X$ such that $q_\#(\pi_1(\tilde X))=ker\theta$.
	Consider the $n!$-sheeted normal covering space $\pi:F_n(Y)\rightarrow D_n(Y)$ and note that the image of $\pi_1(\tilde X)$ under $\Phi_\#\circ q_\#$ is contained in $P_n(Y)$. 
	Then, according to the lifting criterion, there is a lift $\hat \Phi:\tilde X\rightarrow F_n(Y)$ of $\Phi\circ q$, which we can write as $\hat \Phi=(f_1,\ldots, f_n)$, where $f_1,\ldots f_n:\tilde X\rightarrow Y$ are suitable (single-valued) maps.
	We obtain the following commutative diagram;
	\begin{multicols}{2}
		\noindent
		\begin{equation*}
			\begin{tikzcd}[column sep=huge]
				&&F_n(Y)\arrow{d}{\pi}\\
				\tilde X \arrow{r}[swap]{q}\arrow[dashed]{urr}{\hat \Phi} &X\arrow{r} [swap]{\Phi}&D_n(Y).
			\end{tikzcd}
		\end{equation*}
		
		\noindent
		\begin{equation*}
			\begin{tikzcd}[column sep=huge]
				\pi_1(\tilde X) \arrow[dashed]{r}{\hat\Phi_\#}\arrow{d}[swap]{q_\#}&P_n(X)\arrow{d}{\pi_\#}\\
				\pi_1(X) \arrow{r}[swap]{\Phi_\#}&B_n(X)
			\end{tikzcd}
		\end{equation*}
	\end{multicols}
	
	Consider the group of deck-transformations of these two coverings $q:\tilde X\rightarrow X$ and $\pi:F_n(Y)\rightarrow D_n(Y)$, which we denote by $\mathcal Deck(\tilde X, q)$ and $\mathcal Deck(F_n(Y),\pi)$, respectively.
	Since $ker\theta$ and $P_n(Y)$ are normal subgroups of $\pi_1(X)$ and $B_n(Y)$, respectively, it follows by \cite[p. 71, Proposition 1.39]{hatcher2002algebraic} that 
	$\mathcal Deck(\tilde X, q)\simeq\pi_1(X)/ker\theta$ 
	and $\mathcal Deck(F_n(Y),\pi)\simeq B_n(Y)/P_n(Y)$, which is isomorphic to $\mathcal S_n$.
	We denote the image of $\theta$ by $L'$, which is a subgroup of $\mathcal S_n\simeq\mathcal Deck(F_n(Y),\pi)$.
	By the first isomorphism theorem it follows that $\mathcal Deck(\tilde X,q)$ is isomorphic to $L'=im \theta$.
	We denote this isomorphism by $\Gamma:\mathcal Deck(\tilde X,q)\rightarrow L'$ and $\Gamma'=\iota\circ \Gamma$, where $\iota:L'\hookrightarrow \mathcal S_n$ is the inclusion.
	By construction $\hat\Phi$ is equivariant in relation to $\Gamma$, i.e. for every $\delta\in \mathcal Deck(\tilde X,q)$ we have
	\begin{equation}\label{Phidelta=sigmaPhi}
		\hat\Phi\circ \delta= \Gamma(\delta)\cdot\hat\Phi,
	\end{equation}
	where $\cdot$ denotes the action of $L'\subset \mathcal S_n$ on $F_n(Y)$.
	
	\begin{notat}\label{notation non-split}
		Summarizing the notation of the last paragraphs we have, for an $n$-valued non-split map $\phi:X\multimap Y$ and correspondent map $\Phi:X\rightarrow D_n(Y)$, the following concepts;
		\begin{enumerate}
			\item $\theta=\rho\circ\Phi_\#:\pi_1(X)\rightarrow \mathcal S_n$, where $\rho:B_n(Y)\rightarrow \mathcal S_n$ takes a braid to its underlying permutation,
			\item $ker \theta = \Phi_\#^{-1}(P_n(Y))$ is a (normal) subgroup of $\pi_1(X)$;
			\item $q:\tilde X\rightarrow X$ is the (normal/regular and finite) covering space induced by $ker \theta$;
			\item $\hat \Phi=(f_1,\ldots, f_n):\tilde X\rightarrow F_n(Y)$ is a lift of $q\circ\Phi: \tilde{X}\rightarrow D_{n}(Y)$;
			\item $\mathcal Deck(\tilde X,q)\simeq \pi_1(X)/ker \theta$ and $\mathcal Deck(F_n(Y),\pi)\simeq\mathcal S_n$ are the (finite) groups of deck-transformations of the covering spaces $q:\tilde X\rightarrow X$ and $F_n(Y)\rightarrow D_n(Y)$, respectively;
			\item $L'=im\theta$ is a subgroup of $\mathcal Deck(F_n(Y),\pi)\simeq \mathcal S_n$ and consequently acts on $F_n(Y)$;
			\item $\Gamma:\mathcal Deck(\tilde X,q)\rightarrow L'$ is the isomorphism induced by $\theta$ and $\Gamma'=\iota \circ\Gamma$, where $\iota:L'\hookrightarrow \mathcal S_n$ is the inclusion.
		
		\end{enumerate}
	
		Furthermore the following diagrams are commutative; 
		\begin{multicols}{2}
			\noindent
			\begin{equation}\label{eq:classiclift_nonsplit}
				\begin{tikzcd}[column sep=huge]
					\tilde X \arrow{d}[swap]{q}\arrow{r}{\hat \Phi=(f_1,\ldots,f_n)} &F_n(Y)\arrow{d}{\pi}\\
					X\arrow{r}[swap]{\Phi}&D_n(Y).
				\end{tikzcd}
			\end{equation}
			
			\noindent
			\begin{equation}\label{diagrama n-valued <-> BU groupaction}
				\begin{tikzcd}[column sep=huge]
					\pi_1(\tilde X) \arrow{d}[swap]{q_\#}\arrow{r}{\hat \Phi_\#=(f_1,\ldots,f_n)_\#} &P_n(Y)\arrow{d}{\pi_\#}\\
					\pi_1(X)\arrow{r}[swap]{\Phi_\#}\arrow{d}&B_n(Y)\arrow{d}{\rho}\\
					\mathcal Deck(\tilde X,q)\arrow{r}[swap]{\Gamma'=\iota \circ \Gamma}&\mathcal S_n.
				\end{tikzcd}
			\end{equation}
			
		\end{multicols}	
	\end{notat}
	
	The result of the following proposition, connecting the lift factors $f_1,\ldots ,f_n$ and the deck-transformations of both covering spaces $q:\tilde X\rightarrow X$ and $\pi:F_n(Y)\rightarrow D_n(Y)$ is a crucial tool for all of the coming three sections.
	This fact is already known, but is written in a more explicit format that is convenient for the present paper.
	\begin{prop}\label{fsigmak=fdeltak}
		Let $\phi:X\multimap Y$ be an $n$-valued non-split map with correspondent map $\Phi:X\rightarrow D_n(Y)$.
		Let $\hat\Phi=(f_1,\ldots,f_n):\tilde X\rightarrow F_n(Y)$ be a lift of $q\circ \Phi$, as in Notation \ref{notation non-split}.
		Then, for every deck transformation $\delta\in \mathcal Deck(\tilde X,q)$ with associated permutation $\tau=\Gamma(\delta)$ and for every $i\in\{1,\ldots,n\}$ and $k\in \mathbb N$, it holds that 
		\begin{equation*}
			f_{\tau^{k}(i)}=f_i\circ\delta^{k}.
		\end{equation*}   
	\end{prop}
	\begin{proof}
		Let $\delta\in \mathcal Deck(\tilde X,q)$ and $\tau=\Gamma(\delta)$, then from Equation \eqref{Phidelta=sigmaPhi} it follows that
		\begin{equation*}
			(f_{\tau(1)},\ldots,f_{\tau(n)})=\tau\cdot\hat\Phi=\hat\Phi\circ \delta=
			(f_1\circ \delta,\ldots,f_n\circ \delta)
		\end{equation*}
		Hence $f_{\tau(i)}=f_i\circ\delta$, for every $i \in \{1,\ldots, n\}$.
		Iterating this $k$ times provides us with the wished equation.
	\end{proof}

\section{Lift factors and group action by stabilizer groups}\label{sec:lift factors}
	In Section \ref{sec:prelim} we described the approach given in \cite{brown2023lift}, where, for every $n$-valued non-split map $\phi:X\multimap Y$ and its correspondent map $\phi:X\rightarrow D_{n}(Y)$, a covering space $q:\tilde X\rightarrow X$ is constructed, together with $\hat\Phi=(f_1,\ldots,f_n):\tilde X\rightarrow F_n(Y)$, which is a lift of $\Phi\circ q$.
	Given a lift $\hat\Phi=(f_1,\ldots,f_n)$, it holds that every lift of $\Phi\circ q$ to $F_n(Y)$ is given in the form $(f_{\tau(1)},\ldots,f_{\tau(n)}):\tilde X\rightarrow F_n(Y)$, where $\tau\in \mathcal S_n$ is a permutation (see \cite{brown2023lift}[Theorem 2.1 ]).
	Hence the maps $f_1,\ldots,f_n$ are independent of the choice of the lift and are called the \textit{lift factors} of $\Phi:X\rightarrow D_{n}(Y)$.
	In this section we study the effect of the group action $\mathcal Deck(\tilde X,q)\simeq L'\subseteq \mathcal S_n$ on the lift factors $f_1,\ldots, f_n$.
	We select certain subgroups of $\mathcal Deck(\tilde X,q)$, respectively of $L'$, and consider several intermediate covering spaces (see Section \ref{sec:prelim} for a definition) of $q:\tilde X\rightarrow X$, respectively of $\pi:F_n(Y)\rightarrow D_n(Y)$, and maps between these intermediate covering spaces.

	Note that for any proper subgroup $G\subset \mathcal Deck(\tilde X,q)$ it holds that $\tilde X/G$ is an intermediate covering space of $q:\tilde X\rightarrow X$ and we denote the covering space map between $\tilde X/G$ and $X$ by $q_2$.
	Considering the isomorphism $\Gamma:\mathcal Deck(\tilde X,q)\rightarrow L'\subseteq \mathcal S_n$ given in Notation \ref{notation non-split} it holds that 
	$\hat \Phi=(f_1,\ldots,f_n):\tilde X\rightarrow F_n(Y)$ induces a map on the quotient spaces from $\tilde X/G$ to $F_n(Y)/\Gamma(G)$.
	Note that there is no lift of $\Phi\circ q_2$ to $F_n(Y)$, i.e. $\Phi\circ q_2$ does not split into single-valued maps.
	Nevertheless, as long as $\Gamma(G)$ has more than one orbit when acting on $\{1,\ldots,n\}$, we can describe “partial ways to split” the map $\Phi\circ q_2$.
	To give an idea of the concept used in this section, we suppose that $G$ is such that $\Gamma(G)$ has $s$ orbits of consecutive numbers with $s\geq 2$, i.e. $\Gamma(G)\subseteq \mathcal S_{k_1}\times\cdots \times\mathcal S_{k_s}$, for some $k_1,\ldots, k_s\in \{1,\ldots,n\}$ with $k_1+\cdots+ k_s=n$ and $k_1\neq n$.
	If there is more than one option, for $k_1,\ldots, k_s$, we pick the one, where $s$ is the biggest number.
	The general construction, where the orbits of $\Gamma(G)$ do not contain only consecutive numbers, is detailed in Subsection \ref{D_P}.

	Since $\Gamma(G)\subseteq \mathcal S_{k_1}\times\cdots \times\mathcal S_{k_s}$, it holds that $F_n(Y)/\Gamma(G)\rightarrow F_n(Y)/(\mathcal S_{k_1}\times\cdots \times\mathcal S_{k_s})$ is a covering space.
	While $F_n(Y)/\Gamma(G)$ is not necessarily known, the space $F_n(Y)/(\mathcal S_{k_1}\times\cdots \times\mathcal S_{k_s})$, where each $\mathcal S_{k_i}$ permutes the $k_i$ points of the $i$-th block, is known in the literature and called a \textit{mixed configuration space} or an \textit{intermediate configuration space}, typically denoted by $D_{k_1,\ldots,k_s}(Y)$.
	Abusing notation, we denote the projections to the quotient space between configuration spaces (ordered, mixed and unordered) all by $\pi$, the same symbol as we are already using for the projection $\pi:F_n(Y)\rightarrow D_n(Y)$.
	Then there exists a map $\psi:\tilde X/G\rightarrow D_{k_1,\ldots,k_s}(Y)$, such that the following diagram commutes;
		\begin{equation}\label{int cov space G}
		\begin{tikzcd}[column sep=huge]
			\tilde X\arrow{r}{\hat \Phi=(f_1,\ldots,f_n)}\arrow{dd} & F_n(Y)\arrow{d}\arrow[bend left=90]{dd}{\pi}\\
						&F_n(Y)/\Gamma(G)\arrow{d}\\
			\tilde X/G \arrow[dashed,swap]{r}{\psi=(\psi_1,\ldots,\psi_s)}\arrow{d}[swap]{q_2} \arrow[dashed]{ur} & D_{k_1,\ldots,k_s}(Y)\arrow{d}{\pi}\\
			X \arrow[swap]{r}{\Phi} & D_n(Y).
		\end{tikzcd}
	\end{equation}

	When we compose the lift $\psi$ with the projections $D_{k_1,\ldots,k_s}(Y)\rightarrow D_{k_r}(Y)$, we obtain, for each $r\in \{1,\ldots,s\}$ a map $\psi_r:\tilde X/G\rightarrow D_{k_r}(Y)$.
	It holds, for every $x\in X$, that $\psi(x)=(\psi_1(x),\ldots,\psi_s(x))$ and for two distinct $r,t\in \{1,\ldots,s\}$ it holds that $\psi_r(x)\cap \psi_t(x)=\emptyset$.
	In this case we write that $\psi=(\psi_1,\ldots,\psi_s)$.
	Hence on the intermediate covering space $\tilde X/G$ the map splits into $s$ multivalued maps $\psi_1,\ldots \psi_s$, where $\psi_r$ is a $k_r$-valued map, for every $r\in \{1,\ldots,s\}$.
	
	In Subsection \ref{D_P} we show how to construct such a splitting in multivalued maps, for any subgroup $G$ of $\mathcal Deck(\tilde X,q)$, where $\Gamma(G)$ has more than one orbit.
	In this current work we focus on specific subgroups of $\mathcal Deck(\tilde X,q)$, namely those, where at least one of the maps $\psi_1,\ldots, \psi_s$ is a single-valued map, i.e. where a single-valued map "splits off".
	This is the case for $L_i=\Gamma^{-1}(L'_i)$, where $L'_i=\{\tau\in L':\tau(i)=i\}$ is the stabilizer group of $i\in \{1,\ldots,n\}$ (see Notation \ref{notation action Li}).
	Indeed, $L'_i\subseteq \mathcal S_{i-1}\times \{i\}\times \mathcal S_{n-i}$, therefore there exists a map $\psi=(\psi_1,f_i,\psi_3):\tilde X/L_i\rightarrow D_{i-1,1,n-i}(Y)$, such that the diagram analogous to Diagram \eqref{int cov space G} commutes.
	Furthermore we show in Lemma \ref{int cov L fi=> G<Li}, that $L_i$ is the biggest subgroup of $\mathcal Deck(\tilde X,q)$, such that $f_i$ "splits off", i.e. such that $f_i$ is one of the maps $(\psi_1,\ldots,\psi_s)$.
	
	We divide this section in five subsections;
	in the first one, Subsection \ref{subsec:partitions, Psplits}, we address different partitions of the integer $n$, respectively the set $\{1,\ldots,n\}$, including those, where at least one subset does not only contain consecutive numbers.
	Furthermore we define $\mathcal P$-splits, which helps us to describe certain concepts of Diagram \eqref{int cov space G}.
	Subsequently in Subsection \ref{subsec:intermediate lift factors} we analyze the stabilizer groups $L'_i$, defined in Notation \ref{notation action Li}, and their effects on the lift factors $f_i$. 
	We show that $f_i$ induces a map on $\tilde X/L_i\rightarrow X$ and that $L_i$ is the biggest subgroup of $\mathcal Deck(\tilde X,q)$ with this property. 
	In Subsection \ref{subsec:tower} we define various intermediate covering spaces, of both the covering space $q:\tilde X\rightarrow X$, as well as intermediate configuration spaces for the covering space $\pi:F_n(Y)\rightarrow D_n(Y)$, based on the stabilizer groups $L_{\mathbb O_{i_s}}$ of an orbit $\mathbb O_{{i_s}}$.
	With these definitions we can build a tower, with Diagram \ref{int cov space G} as inspiration for the levels, that shows at which step of intermediate covering space each lift factor $f_i$ "splits off".
	We also offer an example to help with the understanding of the constructed tower.
	In Subsection \ref{subsection:Li}, we start to study, again illustrating with an example, what happens to the stabilizer groups $L_{i_s}$, especially if $L_{i_s}$ is not equal to the stabilizer of the orbit $L_{\mathbb O_{i_s}}$, which happens when $L_{i_s}$ is not normal.
	This is an aspect that still merits further investigation.
	In the last subsection, Subsection \ref{subsection: regular groups}, we recall some concepts of group theory and apply them in the context of this section.
	\color{black}
	
	\subsection{Partitions of $n$ and $\mathcal P$-splits}\label{subsec:partitions, Psplits}
		
		 
			In this section we look closer to partitions of the set $\{1,\ldots,n\}$.
			This allows us to explore broader subgroups of $\mathcal Deck(\tilde X,q)$, than the ones in the beginning of this section and in Diagram \eqref{int cov space G}.
			Subsequently we use our considerations about partitions to define, for an $n$-valued map and certain partitions $\mathcal P$, a lifting which we call a $\mathcal P$-split.
		\color{black}
		
		Let $\phi:X\rightarrow Y$ be an $n$-valued map with correspondent map $\Phi:X\rightarrow D_n(Y)$ and let $q:\tilde X\rightarrow X$, $\hat \Phi:\tilde X\rightarrow F_n(Y)$, $L'\subseteq \mathcal S_n$ and $\Gamma:\mathcal Deck(\tilde X,q)\stackrel{\sim}{\rightarrow} L'$ be as in Notation \ref{notation non-split}. 
		Since $\mathcal S_n$  acts freely on $F_n(X)$, so does the subgroup $L'$.
		But we can also look at how $\mathcal S_n$, respectively $L'$, acts on $\{1,\ldots,n\}$.
		Let $\mathbb O_i$ be the orbit of $i$ under $L'$, i.e. $\mathbb O_i=\{j\in \{1,\ldots,n\}: \exists \tau\in L' \text{ s.t. } \tau(i)=j \}$.
		We now pick the smallest element of each orbit and define a set of these elements $\mathbb I_0=\cup_{i=1}^n min \{j\in \mathbb O_i\}=\{i_1,\ldots,i_r\}$.
		For a subset $\mathbb I\subseteq \mathbb I_0$ we denote $\mathbb O_{\mathbb I}=\bigcup_{i\in \mathbb I} \mathbb O_i$ as being the union of the orbit of each integer in $\mathbb I$.
		Let furthermore, for each $i_s\in \mathbb I_0$, the integer $m_{s}=|\mathbb O_{i_s}|$ be the length of the orbit $\mathbb O_{i_s}$.
		
		Note that in the construction of $L'$ we made a choice of ordered label on the base point $\{p_1,\ldots,p_n\}$ in $D_n(Y)$, which defines the permutation in $L'$.
		Since $D_n(Y)=F_n(Y)/\mathcal S_n$ there is no order to the set, hence we suppose that we labeled the set $\{p_1,\ldots,p_n\}$ in a way such that each orbit is composed of subsequent numbers, i.e. for each $i\in \mathbb I_0$ it holds that $\mathbb O_{i}=\{i,i+1,\ldots,i+m_i-1\}$.
		In that case it holds that $L'$ is a subgroup of $\mathcal S_{m_1}\times \ldots\times \mathcal S_{m_r}\subseteq \mathcal S_n$. 
		In the current work we assume for every $n$-valued map $\phi:X\stackrel{n}{\multimap}Y$, that we ordered the base points $\{p_1,\ldots,p_n\}$, such that each orbit consists of consecutive numbers.
		
		Consider that the orbits $\{\mathbb O_{i}: i\in \mathbb I_0\}=\{\mathbb O_{i_1},\ldots,\mathbb O_{i_r}\}$ determine a partition of the set $\{1,\ldots,n\}$, i.e.  $\bigcup_{i\in \mathbb I_0}\mathbb O_i=\{1,\ldots,n\}$ and the orbits $\mathbb O_{i_s}$ are pairwise disjoint.
		We denote this partition by $\mathcal P_{\mathbb I_0}=\{\mathbb O_{i_1},\ldots,\mathbb O_{i_r}\}$.
		Note that $P_{\mathbb I_0}$ induces trivially a partition of the integer $n$, given by $(m_1,\ldots,m_r)=(|\mathbb O_{i_1}|,\ldots,|\mathbb O_{i_r}|)$, hence if we are using this aspect of the partition, we write, abusing notation, that $\mathcal P_{\mathbb I_0}=(m_1,\ldots,m_r)$.
		Recall that we can partially order partitions, indeed, if we have two partitions $\mathcal P$ and $\hat P$ of $\{1,\ldots,n\}$, we say that $\hat {\mathcal P}$ is \textit{finer than} (or a \textit{refinement of}) $\mathcal P$, if every set in $\hat {\mathcal  P}$ is contained in a set in $\mathcal P$.
		The refinement is a \textit{strict refinement}, if $\hat {\mathcal P}\neq \mathcal P$.
		
		The partition $\mathcal P_{\mathbb I_0}=(m_1,\ldots,m_r)$ has an obvious connection to the intermediate configuration space $D_{m_1,\ldots,m_r}(Y)=F_n(Y)/(\mathcal S_{m_1}\times \ldots\times \mathcal S_{m_r})$ and we write $D_{\mathcal P_{\mathbb I_0}}(Y)=D_{m_1,\ldots,m_r}(Y)$.\label{D_P}
		Note that there are many different subgroups $G$ of $\mathcal S_n$, which are isomorphic to $\mathcal S_{m_1}\times \ldots\times \mathcal S_{m_r}$.
		In this case it clearly holds that $F_n(Y)/G$ is homeomorphic to $F_n(Y)/(\mathcal S_{m_1}\times \ldots\times \mathcal S_{m_r})$, through multiplication by a permutation in $\mathcal S_n$.
		But in the context of $n$-valued map with lifts to $F_n(Y)$ the order of the coordinates matter, therefore we distinguish the different ways to partition $\{1,\ldots,n\}$ in sets of fixed order.
		Let $\mathcal P=\{K_1,\ldots, K_r\}$ be a partition of $\{1,\ldots,n\}$, i.e. $K_1,\ldots, K_r$ are pairwise disjoint subsets of $\{1,\ldots,n\}$ such that $K_1\cup\cdots\cup K_r=\{1,\ldots,n\}$.
		Consider the biggest subgroup that preserves the partition $\mathcal P$, which is given by $H_{\mathcal P}=\{\tau\in \mathcal S_n: \text{ if } k\in K_s \text{ then }\tau(k)\in K_s, \text{ for }s=1,\ldots, r\}$ and is isomorphic to $(\mathcal S_{|K_1|}\times \ldots\times \mathcal S_{|K_r|})$.
		We denote $D_{\mathcal P}(Y)=D_{K_1,\ldots, K_r}(Y)=F_n(Y)/H_{\mathcal P}$ being the \textit{intermediate configuration space}, which is obtained by taking the quotient of $F_n(Y)$ by $H_{\mathcal P}$.
		The fundamental group $\pi_1(D_{\mathcal P}(Y))$ is the \textit{mixed braid group} (see, for example, \cite{gonccalves2025splitting}), denoted by $B_{|K_1|,\ldots,|K_r|}(Y)$, in case that the sets $K_1,\ldots,K_r$ consist of consecutive numbers.
		In this paper, we also denote $\pi_1(D_{\mathcal P}(Y))$ by $B_{\mathcal P}(Y)$; this notation remains precise even when the sets $K_1,\ldots,K_r$ do not consist of consecutive numbers.
		Note that ${\mathcal P}_{\mathbb I_0}=\{\mathbb O_{i_1},\ldots, \mathbb O_{i_r}\}$ and each $\mathbb O_{i_s}$ consists of consecutive numbers, 
		thus it holds that $D_{{\mathcal P}_{\mathbb I_0}}(Y)=D_{|\mathbb O_{i_1}|,\ldots, |\mathbb O_{i_r}|}(Y)=D_{m_1,\ldots, m_r}(Y)$ and $\pi_1(D_{{\mathcal P}_{\mathbb I_0}}(Y))$ is the mixed braid group given by $B_{m_1,\ldots, m_r}(Y)$.
		Therefore, for any subgroup $G$ of $\mathcal Deck(\tilde X,q)$ we can consider the different orbits $\mathcal O_1,\ldots, \mathcal O_r$ of $\Gamma(G)$ acting on $\{1,\ldots,n\}$.
		Let $\mathcal P_G=\{\mathcal O_1,\ldots, \mathcal O_r\}$ be the partition consisting of the orbits of $\Gamma(G)$.
		Then $\Gamma(G)$ preserves the partition $\mathcal P_G$ and is thus a subgroup of the biggest group preserving $\mathcal P_G$, which we denoted by $H_{\mathcal P_G}$.
		Therefore there is a covering space map from $D_{\mathcal P_G}(Y)$ to $F_n(Y)/\Gamma(G)$ giving rise to a diagram analogous to Diagram \eqref{int cov space G}.

		In the following definition we define a concept, which we use for the rest of this section.
		\begin{defi}\label{def:P-split}
			Let $P$ be a partition of $\{1,\ldots,n\}$, and let $\hat {\mathcal P}=\{K_1,\ldots,K_r\}$ be a refinement of ${\mathcal P}$. 
			We say that a map $\Phi:X\rightarrow D_{\mathcal P}(Y)$ $\hat {\mathcal P}$-splits (or $(|K_1|,\ldots,|K_r|)$-splits), if there exists a map $\psi:X\rightarrow D_{\hat {\mathcal P}}(Y)$ such that the following diagram commutes. We call the map $\psi$ a $\hat {\mathcal P}$-split of $\phi$.
		\begin{equation*}
			\begin{tikzcd}[column sep=huge]
				\tilde X\arrow{r}{(f_1,\ldots,f_n)} \arrow[swap]{d}{q}&F_n(Y)\arrow{d}{\pi}\\
				X \arrow[dashed]{r}{\psi}\arrow{rd}[swap]{\Phi}&D_{\hat {\mathcal P}}(Y)\arrow{d}{\pi}\\
				 &D_{\mathcal P}(Y).
			\end{tikzcd}
		\end{equation*}
		\end{defi}

		In the next proposition we construct a ${\mathcal P}$-split $(\psi_1,\ldots, \psi_r)$ for $n$-valued non-split maps with more than one orbit; in the case of a single orbit, i.e. $\mathbb I_0=\{1\}$, the statement holds trivially.
		This result was first stated in \cite{staecker2021partitions}, using different terminology; the maps $\psi_1,\ldots, \psi_r$ are called submaps of $\Phi$, and the set $\{\psi_1,\ldots,\psi_r\}$ is referred to as a partition of $\Phi$ into irreducibles (see \cite[Section 3 and 4]{staecker2021partitions}).
		Here we use the terminology established in the present work.
		\color{black}
		
		\begin{prop}\label{Phi P-splits}
			Let $\Phi:X\rightarrow D_n(Y)$ be an $n$-valued map and let $L'\subseteq \mathcal S_n$ and $q:\tilde X\rightarrow X$ be as in Notation \ref{notation non-split}.
			Let furthermore $\mathbb O_i$ be the orbit of $i\in \{1,\ldots,n\}$ by $L'$, $\mathbb I_0=\{i_1,\ldots,i_r\}$ the set of the smallest number of each orbit and let ${\mathcal P}_{\mathbb I_0}$ be the partition of $\{1,\ldots,n\}$ given by $\{\mathbb O_{i_1},\ldots, \mathbb O_{i_r}\}$.
			Then there exists a lift $\psi=(\psi_1,\ldots,\psi_r):X\rightarrow D_{{\mathcal P}_{\mathbb I_0}}(Y)=D_{m_1,\ldots,m_r}(Y)$ of $\Phi$, i.e. $\Phi$ ${\mathcal P}_{\mathbb I_0}$-splits or, equivalently  $\Phi$ $(m_1,\ldots,m_r)$-splits.
			Furthermore there exists no strict refinement $\hat {\mathcal P}$ of the partition ${\mathcal P}$, such that $\Phi: X\rightarrow D_n(Y)$ $\hat {\mathcal P}$-splits.
			\begin{equation*}
				\begin{tikzcd}[column sep=huge]
					\tilde X\arrow{r}{(f_1,\ldots,f_n)} \arrow[swap]{d}{q}&F_n(Y)\arrow{d}{\pi}\\
					X \arrow[dashed]{r}{\psi=(\psi_1,\ldots,\psi_r)}\arrow{dr}[swap]{\Phi} 
					&D_{m_1,\ldots,m_r}(Y)\arrow{d}{\pi}\\
					&D_{n}(Y).
				\end{tikzcd}
			\end{equation*}
		\end{prop}
		
		\begin{proof}
			Suppose that there is a braid $\beta\in \Phi_\#(\pi_1(X))$ such that $\beta\notin \pi_\#(B_{m_1,\ldots,m_r}(Y))$, then, by construction of $L'=im(\rho\circ\Phi_\#)$ it holds that $\rho(\beta)\in L'$.
			From the explanation of the beginning of this section we know that $L'\subseteq \mathcal S_{m_1}\times \ldots\times \mathcal S_{m_r}\subseteq \mathcal S_n$, which is exactly $\rho\circ \pi_\#(B_{m_1,\ldots,m_r}(Y))$, leading to a contradiction.
			By the lifting criterion there exists a lift such that $\Phi=\pi\circ \psi$.
			
			For the second statement let $\hat P$ be a refinement of the partition $P_{\mathbb I_0}$ of $\{1,\ldots,n\}$.
			Without loss of generality we assume that $\hat P=\{\mathbb O_{i_1},\ldots,\mathbb O_{i_{s-1}},K_1,K_2,\mathbb O_{i_{s+1}}, \ldots ,\mathbb O_{i_r}\}$, where $K_1$ and $K_2$ are disjoint, non-empty sets of consecutive numbers with $K_1\cup K_2=\mathbb O_{i_s}$ and $i_s\in K_1$.
			Let $k_2\in K_2$, since the orbit $\mathbb O_{i_s}=\{i_s,\ldots, i_s+m_s-1\}$ consists of $m_s$ consequent integers there exists $\tau\in L'$ such that $\tau(i_s)=i_s+k_s$, hence $L'$ is not a subset of $\mathcal S_{m_1}\times\cdots \times\mathcal S_{m_{s-1}}\times\mathcal S_{|K_1|}\times \mathcal S_{|K_2|}\times \mathcal S_{m_{s+1}}\cdots \times\mathcal S_{m_{r}}$ and therefore $\Phi$ does not lift to $D_{\hat P}(Y)$.
		\end{proof}
		
		In the proof of Proposition \ref{Phi P-splits}, we used the fact that $L' \subseteq \mathcal{S}_{m_1} \times \ldots \times \mathcal{S}_{m_r}$ to lift the map $\Phi : X \rightarrow D_n(Y)$ to
		$D_{m_1,\ldots,m_r}(Y) = F_n(Y)/(\mathcal{S}_{m_1} \times \ldots \times \mathcal{S}_{m_r})$.
		The same ideas from covering space theory imply that there exists a lift $X \rightarrow F_n(Y)/L'$ of $\Phi : X \rightarrow D_n(Y)$.
		However, since $L'$ can be any subgroup of $\mathcal{S}_n$, the space $F_n(Y)/L'$ may be completely unexplored.
		Therefore, we focus on the intermediate configuration spaces, whose fundamental groups are the mixed braid groups.
	
	\subsection{Intermediate lift factors}\label{subsec:intermediate lift factors}
		
	Before we go back to $n$-valued maps, we recall some (basic) concepts of group theory.
	Consider, for every $i\in \{1,\ldots,n\}$ the subgroup $L_i'=\{\tau\in L':\tau(i)=i\}$, called the \textit{stabilizer or isotropy group of $i$ in $L'$}.
	If instead of a point we fix a subset $J\subseteq\{1,\ldots,n\}$ in the following way $L'_J=\{\tau\in L': \tau(i)=i \textit{ for all } i\in J\}=\bigcap_{i\in J}L_i'$ we obtain the so called \textit{pointwise stabilizer of $J$ in $L'$}. 
	%
	For each $j\in \mathbb O_i$ it holds that $L_j'$ is a conjugate of $L_i'$ by an element of $L'$, more specifically $L_j'=\tau\cdot L_i'\cdot\tau^{-1}$, where $\tau\in L'$ with $\tau(i)=j$.
	Which implies that if $L_i'$ is a normal subgroup of $L'$, then $L_j'=L_i'$ for all $j\in \mathbb O_i$.
	For every $i\in \{1,\ldots,n\}$ it holds that $L'_{\mathbb O_i}$, the pointwise stabilizer of the orbit $\mathbb O_i$ in $L'$, is the normal core, or normal interior of $L_i'$ in $L'$, i.e. $L'_{\mathbb O_i}$ is the biggest normal subgroup of $L'$ contained in $L_i'$.
	Which means that the stabilizer $L_i'$ of a point $i\in \{1,\ldots,n\}$ is normal in $L'$, if and only if, every element that stabilizes $i$, stabilizes the whole orbit $\mathbb O_i$ of $i$ pointwise.
	Consequently $L'_{\mathbb O_{\mathbb I}}=\bigcap_{i\in \mathbb I}L'_{\mathbb O_i}$, the stabilizer of the orbits of a set $\mathbb I$, is the biggest normal subgroup of $L'$ contained in $L'_j$, for every $i\in \mathbb I$ and every $j\in \mathbb O_i$.

	Focusing again on our $n$-valued map $\Phi$, we use the isomorphism $\Gamma:\mathcal Deck(\tilde X,q)\rightarrow L'$ (see Notation \ref{notation non-split}) to define subgroups of $\mathcal Deck(\tilde X,q)$, isomorphic to the stabilizer subgroups in $L'$, setting,
	for every subset $J\subseteq \{1,\ldots,n\}$, the subgroup $L_J=\Gamma^{-1}(L_J')$.
	Observe, that for every $i\in \{1,\ldots,n\}$ and every $\delta\in L_i$, it holds, by Proposition \ref{fsigmak=fdeltak} that $f_i\circ \delta=f_{\tau(i)}=f_i$, where $\tau=\Gamma(\delta)\in L_i'$.
	Therefore $f_i:\tilde X\rightarrow Y$ induces a map on the quotient space $\bar f_i:\tilde X/L_i\rightarrow Y$.
	Also note that the covering map $q:\tilde X\rightarrow X$ induces a covering map on the quotient space $q_i:\tilde X/L_i\rightarrow X$.
	Consider the covering map $p_i:\tilde X\rightarrow \tilde X/L_i$, given by the projection, taking a point $\tilde x\in \tilde X$ to the class $[\tilde x]\in \tilde X/L_i$, then $q=q_i\circ p_i$ and $f_i=\bar f_i\circ p_i$.
	Let $J\subseteq \{1,\ldots,n\}$ and $i\in J$, then $L_J$ is a subgroup of $L_i$ and consequently $f_i:\tilde X\rightarrow Y$ induces a map on the quotient space
	$\tilde X/L_J$, which we, abusing notation, denote again as $\bar f_i:\tilde X/L_J\rightarrow Y$.
	We denote the map on the quotient space induced by $q:\tilde X\rightarrow Y$ as $ q_J:\tilde X/L_J\rightarrow Y$ and the projection as $p_J:\tilde X\rightarrow \tilde X/L_J$.
	Again, both $q_J$ and $p_J$ are covering maps 
	and we obtain the following commutative diagrams;
	\begin{multicols}{2}
		
	\noindent	
	\begin{equation}\label{bar qi & bar fi}
		\begin{tikzcd}[column sep=huge]
			\tilde X\arrow{d}[swap]{p_i}\arrow{ddr}{f_i}\arrow[bend right=60,swap]{dd}{q}&\\
			\tilde X/L_i \arrow{d}[swap]{q_i}\arrow[dashed]{dr}[swap]{\bar f_i} \\
			X&Y
		\end{tikzcd}
	\end{equation}
	
		\noindent	
	\begin{equation}\label{LJ bar qi & bar fi}
		\begin{tikzcd}[column sep=huge]
			\tilde X\arrow{d}[swap]{p_J}\arrow{ddr}{f_i}\arrow[bend right=60,swap]{dd}{q}&\\
			\tilde X/L_J \arrow{d}[swap]{ q_J}\arrow[dashed]{dr}[swap]{\bar f_i} \\
			X&Y
		\end{tikzcd}
	\end{equation}
	\end{multicols}
		Note that, in case that $L'_J=L'$, or equivalently $L_J=\mathcal Deck(\tilde X,q)$, then $\tilde X/L_J=X$, $ q_J=id$ and $p_J=q$.
		Also, if $L'_J=\{id\}$, or equivalently $L_J=\{id\}$, then $\tilde X/{L_J}=\tilde X$ and $ q_J=q$ and $p_J=id$.
		
		We now summarize the notation established in the last paragraphs, i.e. in the beginning of this section.
	\begin{notat}\label{notation action Li}
		Summarizing this notation we have the following concepts, for an $n$-valued non-split map $\phi:X\multimap Y$ with correspondent map $\Phi:X\rightarrow D_n(Y)$, using the notation of Notation \ref{notation non-split};
		\begin{enumerate}
			\item $\mathbb O_i=\{j\in \{1,\ldots,n\}: \exists \tau\in L' \text{ s.t. } \tau(i)=j \}$ is the orbit of $i$ by the action $L'$, where we suppose that the base points $p_1,\ldots, p_n$ in $Y$ are ordered in such a way that each orbit $\mathbb O_i$ consists of consecutive numbers only; 
			\item $\mathbb I_0=\bigcup_{i=1}^n min \{j\in \mathbb O_i\}=\{i_1,\ldots,i_r\}$ is the set obtained by picking the smallest element of each orbit;
			\item $\mathbb O_{\mathbb I}=\bigcup_{i\in \mathbb I} \mathbb O_i$, where $\mathbb I\subseteq \mathbb I_0$;
			\item for every $i\in \mathbb I_0$, let $m_i=|\mathbb O_i|$ be the length of the orbit;
			\item $P_{\mathbb I_0}=\{\mathbb O_{i_1},\ldots, \mathbb O_{i_r}\}$ is the partition of $\{1,\ldots,n\}$ determined by $\mathbb I_0=\{i_1,\ldots,i_r\}\subseteq \{1,\ldots,n\}$, but we sometimes abuse notation and write $P_{\mathbb I_0}=(m_1,\ldots,m_r)=(|\mathbb O_{i_1}|,\ldots,|\mathbb O_{i_r}|)$ instead;
			\item $L'_i=\{\tau\in L':\tau(i)=i\}$ the stabilizer subgroup of the index $i\in \{1,\ldots,n\}$ in $L'$ and
			$L'_J=\{\tau\in L': \tau(i)=i \textit{ for all } i\in J\}=\bigcap_{i\in J}L_i'$ the pointwise stabilizer of a subset $J\subseteq \{1,\ldots,n\}$ in $L'$;
			\item $L_i=\Gamma^{-1}(L'_i)$ and $L_J=\Gamma^{-1}(L'_J)$ are subgroups of $\mathcal Deck(\tilde X,q)$, where $i$ is an element and $J$ a subset of $\{1,\ldots,n\}$ and $L=\Gamma ^{-1}(L')=\mathcal Deck(\tilde X,q)$;		
			\item for every $J\subseteq \{1,\ldots, n\}$ let $p_J:\tilde X\rightarrow \tilde X/L_J$ and $q_J:\tilde X/L_J\rightarrow X$, be the
				maps to, respectively, on the quotient space induced by $q:\tilde X\rightarrow X$. If $J=\{i\}$, then we also write $p_i$ and $q_i$ respectively;
			\item for every $i\in\{1,\ldots, n\}$ let $\bar f_i:\tilde X/L_i\rightarrow Y$ be the map on the quotient space induced by $f_i:\tilde X\rightarrow Y$. Abusing notation, we use the same notation $\bar f_i:\tilde X/L_J\rightarrow Y$, for the map induced on the quotient space $\tilde X/L_J$ by $f_i$, where $i\in J$.
		\end{enumerate}
	\end{notat}

		The following lemma shows that $L_i$ is the biggest group (with the order given by inclusion) such that $f_i$ induces a map on the quotient space $\tilde X/L_i$.
	
	\begin{lem}\label{int cov L fi=> G<Li}
		Let $\tilde X/G$ be an intermediate covering space of $q:\tilde X\rightarrow X=\tilde X/L$,  i.e. $G$ is a subgroup of $L$ with covering maps $p:\tilde X\rightarrow \tilde X/G$ and $\tilde X/G\rightarrow X$ and let $g:\tilde X/G\rightarrow Y$ be a map such that $g\circ p=f_i$, then $G$ is a subgroup of $L_i$ and we obtain the following commuting diagram;
		\begin{equation*}
			\begin{tikzcd}[column sep=huge]
				\tilde X\arrow{d}[swap]{p}\arrow{ddr}{f_i}&\\
				\tilde X/G\arrow[dashed]{d}\arrow{dr}[swap]{g}\arrow[bend right=60]{dd}&\\
				\tilde X/L_i \arrow{d}[swap]{q_i}		\arrow{r}[swap]{\bar f_i}&Y\\
				X&
			\end{tikzcd}
		\end{equation*}
	\end{lem}
	\begin{proof}	
		Let $\delta\in G\subseteq L$, then, since $g$ is a (well-defined) map, it holds for every $\tilde x\in \tilde X$ that $f_i(\tilde x)=f_i(\delta(\tilde x))$.
		By Proposition \ref{fsigmak=fdeltak} it holds that $f_i(\delta(\tilde x))=f_{\tau(i)}(\tilde x)$, where $\tau=\Gamma(\delta)$.
		This implies that $\tau(i)=i$, hence $\tau\in L_i'$ and $\delta\in L_i$.
		Consequently $G$ is a subgroup of $L_i$ and by \cite[page 81]{hatcher2002algebraic} determining a composition of covering spaces $\tilde X\rightarrow \tilde X/G\rightarrow \tilde X/L_i\rightarrow X=\tilde X/L$.
	\end{proof}

	\subsection{A tower of lifts}\label{subsec:tower}

	In this subsection, we use the results on lifts and stabilizers from the previous subsections to construct a diagram—a “tower”—that indicates at which intermediate quotient spaces the lift factors $f_1,\ldots, f_n$ ``split off''.
	At the beginning of this section, we saw that for any $G\subseteq \mathcal Deck(\tilde X, q)$ for which $\Gamma(G)$ has more than one orbit, one can find an appropriate map $\psi$ from $\tilde X/G$ to an intermediate configuration space (see Diagram \eqref{int cov space G}).
	By Lemma \ref{int cov L fi=> G<Li} the stabilizer group $L_i'$ is the largest subgroup of $L'$ for which $f_i$ induces a map on the quotient space $\tilde X/L_i$, where $L_i=\Gamma^{-1}(L_i')$ and $i\in \{1,\ldots,n\}$.
	Note that if $L_i'$ is normal in $L'$, then $L_i'$ fixes all integers in the orbit $\mathbb O_i$, i.e. $L_i'=L'_{\mathbb O_i}$.
	More generally $L'_{\mathbb O_i}=\bigcup_{j\in \mathbb O_i}L'_j$, hence for any $j\in \mathbb O_i$ it holds that $f_j$ induces a map on the quotient space $\tilde X/L_{\mathbb O_i}$.
	
	In this subsection we build a tower based on the intermediate covering spaces induced by the stabilizer groups of whole orbits and their intersections, illustrating the largest subgroup $G$ of $\mathcal Deck(\tilde X,q)$, inducing an intermediate covering space $\tilde X/G$, where all lift factors of an orbit "split off".
	Recall that we have $r$ orbits $\mathbb O_{i_1},\ldots, \mathbb O_{i_r}$, where $\mathbb I_0=\{i_1,\ldots,i_r\}$ is the set consisting of the smallest integer of each orbit.
	Recall also that $m_s=|\mathbb O_{i_s}|$ denotes the size of each orbit, for any $s\in \{1,\ldots,r\}$ (see Notation \ref{notation action Li}).
	
	To illustrate the process carried out in this subsection, we consider $\mathbb O_{1}$, the orbit of the integer $1$.
	Note that
	$L'_{\mathbb O_{1}}=\{\tau\in L' : \tau(j)=j \text{ for all } j\in \mathbb O_{1}\}$
	fixes every element of the orbit, hence, $L'_{\mathbb O_{1}} \subseteq \{1\}\times\{2\}\times\ldots\times \{m_1\}\times \mathcal S_{i_2}\times \cdots \times \mathcal S_{i_r}$.
	Considering the covering space map $q_{\mathbb O_{1}}:\tilde X/L_{\mathbb O_1}\rightarrow X$ given in Notation \ref{notation action Li}, there is a $\mathcal P$-split of $\phi\circ q_{\mathbb O_{1}}$ (and of $\psi\circ q_{\mathbb O_{1}}$), where $\mathcal P$ is the partition $\{\{1\},\ldots,\{m_1\},\mathbb O_{i_2},\ldots,\mathbb O_{i_r}\}$.
	In fact we can describe this map exactly, using the ${\mathcal P}_{\mathbb I_0}$-split $\psi=(\psi_1,\ldots,\psi_r):X\rightarrow D_{\mathbb I_0}(Y)$ given in Proposition \ref{Phi P-splits} and the maps $\bar f_1,\ldots,\bar f_{m_1}:\tilde X/L_{\mathbb O_1}\rightarrow Y$ induced on the quotient space by $f_1,\ldots,f_{m_1}$, respectively (see Notation \ref{notation action Li}).
	We obtain the following commutative diagram;
			\begin{equation}\label{int cov space L_O1}
				\begin{tikzcd}[column sep=120]
					\tilde X\arrow{r}{\hat \Phi=(f_1,\ldots,f_n)}\arrow{d} & F_n(Y)\arrow{d}\arrow{d}{\pi}\\
					\tilde X/L_{\mathbb O_1} \arrow[dashed]{r}{(\bar f_1,\ldots,\bar f_{m_1},\psi_{2}\circ  q_{\mathbb O_{1}},\ldots, \psi_r \circ  q_{\mathbb O_{1}})}\arrow{d}[swap]{q_{\mathbb O_{1}}} & D_{1,\ldots,1,m_2,\ldots,m_r}(Y)\arrow{d}{\pi}\\
					X \arrow[swap]{rd}{\Phi} \arrow{r}{\psi=(\psi_1,\ldots,\psi_r)} &D_{m_1,\ldots,m_r}(Y)\arrow{d}{\pi} \\
					& D_n(Y).
				\end{tikzcd}
			\end{equation}
	We can carry out an analogous construction for the stabilizer of two orbits, for example those corresponding to $i_1=1$ and $i_2$, namely
	$	L'_{\mathbb O_{\{i_1,i_2\}}}= L'_{\mathbb O_1\cup \mathbb O_{i_2}}= L'_{\mathbb O_1}\cap L'_{\mathbb O_{i_2}}$.
	More generally, for any subset $\mathbb I\subseteq \mathbb I_0=\{i_1,\ldots,i_r\}$, we consider the stabilizer group $L'_{\mathbb O_{\mathbb I}}$, defined as the subgroup of $L'$ that fixes every element in each orbit $\mathbb O_{i_s}$ with $i_s\in \mathbb I$ (see Notation \ref{notation action Li}).
	
	This correspondence — assigning to each subset of orbits a stabilizer subgroup, an intermediate quotient space, and a partition of the set $\{1,\ldots,n\}$ — forms the basis of the tower construction.
	In the next paragraphs, we define, for each subset $\mathbb I\subseteq \mathbb I_0$, a map from $\tilde X/L_{\mathbb O_{\mathbb I}}$ to a corresponding intermediate configuration space.
	To do this, we first define the appropriate configuration spaces by introducing the corresponding partitions.
	
	For a subset $\mathbb I \subseteq \mathbb I_0$, we define a partition $\mathcal P_{\mathbb I}$ of $n$.
	We begin with the case $\mathbb I=\{i_s\}$.
	In this case, we set
	$\mathcal P_{\{i_s\}}=\{\{1\},\{2\},\ldots,\{i_s-1\},\mathbb O_{i_s},\{i_s+1\},\ldots,\{n\}\},	$
	or, abusively,
	$	\mathcal P_{\{i_s\}}=(1,\ldots,1,m_s,1,\ldots,1)
	=(\mathbf{1}^{m_1+\cdots+m_{s-1}},\, m_s,\, \mathbf{1}^{m_{s+1}+\cdots+m_r}),	$
	where the entry $1$ appears $m_1+\cdots+m_{s-1}$ times before $m_s$ and $m_{s+1}+\cdots+m_r$ times after $m_s$.
	\color{black}
	We define, for each ${\mathbb I}\subseteq \mathbb I_0$ a partition $\mathcal P_{\mathbb I}$ in the analogous way as above, for each $i_s\in {\mathbb I}$ we add the set $\mathbb O_{i_s}$ and in case that $i_s\notin {\mathbb I}$, add the unitary sets $\{j\}$, for all $j\in \mathbb O_{i_s}$, which means for the index notation that we write, for each $i_s\in {\mathbb I}$ the index $m_s$ and in case that $i_s\notin {\mathbb I}$,  we write $m_s$ times the number $1$.
	Hence, for $\mathbb I=\mathbb I_0\setminus\{i_s\}=\{i_1,\ldots, i_{s-1},i_{s+1},\ldots, i_r\}$, where $s\in \{1,\ldots,r\}$, then
	$\mathcal P_{\mathbb I}=\{\mathbb O_{i_1},\ldots,\mathbb O_{i_{s-1}},\{i_s\},\ldots, \{i_{s+1}-1\},\mathbb O_{i_{s+1}},\ldots,\mathbb O_{i_r}\}$, or $\mathcal P_{\mathbb I}=(m_1,\ldots,m_{s-1},\mathbf{1}^{m_s},m_{s+1},\ldots,m_r)$.
	
	Each partition $\mathcal P=\{K_1,\ldots,K_s\}$ of $\{1,\ldots,n\}$ induces the intermediate configuration space $D_{\mathcal P}(Y)=D_{K_1,\ldots,K_s}(Y)$, as explained right before Definition \ref{def:P-split}.
	For a partition ${\mathcal P}_{\mathbb I}$ induced by ${\mathbb I}\subseteq I_0$ we write either $D_{{\mathcal P}_{\mathbb I}}(Y)$ or, to simplify notation, $D_{\mathbb I}(Y)$, which means if ${\mathbb I}=\{i_s\}\subseteq \mathbb I_0$ we have that $D_{\mathbb I}(Y)$ and $D_{\{i_s\}}(Y)$ are notation variations of the space $D_{{\mathcal P}_{\mathbb I}}(Y)$.
	
	The next step is to define our "intermediate ${\mathcal P}$-splits", using, for each $s\in \{1,\ldots,r\}$ the map $\psi_s:X\rightarrow D_{m_s}(Y)$ obtained in Proposition \ref{Phi P-splits}.
	Let $s\in \{1,\ldots, r\}$ and consider the map $\psi_s\circ q:\tilde X\rightarrow D_{m_s}(Y)$.
	Then, since $\psi_s$ is a map on $X=\tilde X/L$, the quotient map $\bar \psi_s:\tilde X/L_J\rightarrow D_{m_s}(Y)$ of $\psi_s\circ q$ is well-defined, for every $J\subseteq \{1,\ldots,n\}$.
	Again we abuse notation and use the notation $\bar \psi_s:X/L_J\rightarrow D_{m_s}(Y)$ independently of the subset $J\subseteq\{1,\ldots,n\}$, the context will make it clear which quotient is the domain of the map $\bar \psi_s$.
	Note that $\bar \psi_s:X/L_J\rightarrow D_{m_s}(Y)$ is equal to $\psi_s\circ q_{J}$, where $q_J$ is as in Notation \ref{notation action Li}.
	We define, for each $\mathbb I\subseteq \mathbb I_0$, the (well-defined) map $\Psi_{\mathbb I}:X/L_{\mathbb O_{\mathbb I_0\setminus\mathbb I}}\rightarrow D_{\mathbb I}(Y)$ by substituting in $(\bar f_1,\ldots,\bar f_n)$, for every $i_s\in\mathbb I$ the $\bar f_{i_s},\bar f_{i_s+1},\ldots,\bar f_{i_s+m_s-1}$ by $\bar\psi_s$. 
	To exemplify let $\mathbb I=\{i_s\}\subseteq \mathbb I_0$, then  $\Psi_{\mathbb I}:\tilde X/L_{\mathbb O_{\mathbb I_0\setminus\mathbb I}}\rightarrow D_{\mathbb I}(Y)= D_{\{i_s\}}(Y)$ is given as
	$\Psi_{\mathbb I}=\Psi_{\{i_s\}}=\{\bar f_1,\ldots, \bar f_{i_s-1},\bar \psi_{s}, \bar f_{i_{s+1}},\ldots, \bar f_{n} \}$.
	For another example let $\mathbb I=\{i_1,\ldots, i_{s-1},i_{s+1},\ldots, i_r\}$, then $\Psi_{\mathbb I}:\tilde X/L_{\mathbb O_{\{i_s\}}}\rightarrow D_{\mathbb I}(Y)$ is given as 
	$\{\bar\psi_1,\ldots, \bar \psi_{s-1}, \bar f_{i_s},\ldots, \bar f_{i_{s+1}-1} ,\bar \psi_{s+1},\ldots, \bar \psi_r\}.$
	We now summarize the notation of the spaces and maps we defined in the last paragraphs.
	\begin{notat}\label{notation maps}
		Summarizing the notation introduced above, we have the following concepts, using the notation established in Notations \ref{notation non-split} and \ref{notation action Li};
		\begin{enumerate}
			\item
				each subset $\mathbb I\subseteq \mathbb I_0$ induces a partition of $\{1,\ldots,n\}$, denoted by ${\mathcal P}_{\mathbb I}$, where ${\mathcal P}_{\mathbb I_0}=\{\mathbb O_{i_1},\ldots,\mathbb O_{i_r}\}$ and ${\mathcal P}_{\mathbb I}$ is a refinement of ${\mathcal P}_{\mathbb I_0}$ obtained by adding, for each $i_s\in {\mathbb I}$ the set $\mathbb O_{i_s}$ and in case that $i_s\notin {\mathbb I}$ the unitary sets $\{j\}$, for all $j\in \mathbb O_{i_s}$;
			\item
				for each partition of $\{1,\ldots,n\}$, given by ${\mathcal P}=\{K_1,\ldots,K_s\}$, let $D_{\mathcal P}(Y)=D_{K_1,\ldots,K_s}(Y)=F_n(Y)/H_{\mathcal P}$, where $H_{\mathcal P}\subseteq \mathcal S_n$ is the biggest subgroup preserving the partition ${\mathcal P}$.
				For a partition ${\mathcal P}_{\mathbb I}$ induced by ${\mathbb I}\subseteq \mathbb I_0$ we write either $D_{{\mathcal P}_{\mathbb I}}(Y)$ or $D_{\mathbb I}(Y)$;
			\item
				for each $\mathbb I\subseteq \mathbb I_0$, we define the map $\Psi_{\mathbb I}:\tilde X/L_{\mathbb O_{\mathbb I_0\setminus\mathbb I}}\rightarrow D_{\mathbb I}(Y)$ by substituting in $(\bar f_1,\ldots,\bar f_n),$ 
				for every $i_s\in \mathbb I$ the part $\bar f_s,\bar f_{s+1},\ldots,\bar f_{s+m_s-1}$ by $\bar\psi_s$, where $\bar\psi_s=\psi_s\circ q_{\mathbb O_{\mathbb I_0\setminus \mathbb I}}$.
		\end{enumerate}
	\end{notat}
	
	By Proposition \ref{Phi P-splits} it holds that $\Phi=\{\psi_1,\ldots,\psi_r\}$ ${\mathcal P}_{\mathbb I_0}$-splits, i.e. $\Phi:X\rightarrow D_n(Y)$ lifts to $(\psi_1,\ldots,\psi_r):X\rightarrow D_{\mathbb I_0}(Y)$, the next proposition gives us a splitting in the intermediate configuration spaces.
	
	\begin{prop}\label{int cov L Psi=> G<Li}
		Let $\phi:X\multimap Y$ be an $n$-valued map with correspondent map $\Phi:X\rightarrow D_n(Y)$ with ${\mathcal P}_{\mathbb I_0}$-split $(\psi_1,\ldots,\psi_r):X\rightarrow D_{\mathbb I_0}(Y)$, giving rise to the whole structure described in Notations \ref{notation non-split},\ref{notation action Li} and \ref{notation maps}.
		Then $(\psi_1,\ldots,\psi_r)\circ q_{\mathbb O_{\mathbb I_0\setminus\mathbb I}}:\tilde X/L_{\mathbb O_{\mathbb I_0\setminus\mathbb I}} \rightarrow D_{\mathbb I_0}(Y)$ ${\mathcal P}_{\mathbb I}$-splits into $\Psi_{\mathbb I}:\tilde X/L_{\mathbb O_{\mathbb I_0\setminus\mathbb I}} \rightarrow D_{\mathbb I}(Y)$.
		Furthermore, for an intermediate covering space $\tilde X\rightarrow\tilde X/G\stackrel{g}\rightarrow X=\tilde X/L$, where $(\psi_1,\ldots,\psi_r)\circ g:\tilde X/G\rightarrow D_{\mathbb I_0}(Y)$ $P_{\mathbb I}$-splits, it holds that $G$ is a subgroup of $L_{\mathbb O_{\mathbb I_0\setminus\mathbb I}}$, which gives rise to the intermediate covering  $\tilde X\rightarrow \tilde X/G\rightarrow \tilde X/L_{\mathbb O_{\mathbb I_0\setminus\mathbb I}}$.
		 \begin{equation*}
		 	\begin{tikzcd}[column sep=huge]
		 		\tilde X\arrow{d}\arrow{r}{(f_1,\ldots,f_n)}\arrow[bend right=60, swap]{dd}{p_{\mathbb O_{\mathbb I_0\setminus\mathbb I}}}&F_n(Y)\arrow{dd}{\pi}\\
		 		\tilde X/G\arrow[dashed]{d}\arrow{dr}\arrow[bend right=60, crossing over]{dd}[swap]{g}		&\\
		 		\tilde X/L_{\mathbb O_{\mathbb I_0\setminus\mathbb I}} \arrow{d}{ q_{\mathbb O_{\mathbb I_0\setminus\mathbb I}}}		\arrow[dashed, swap]{r}{\Psi_{\mathbb I}}&D_{\mathbb I}(Y)\arrow{d}{\pi}\\
		 		X\arrow{r}[swap]{(\psi_1,\ldots,\psi_r)}&D_{\mathbb I_0}(Y)
		 	\end{tikzcd}
		 \end{equation*}
	\end{prop}
	\begin{proof}
		Consider the map $\Psi_{\mathbb I}:\tilde X/L_{\mathbb O_{\mathbb I_0\setminus\mathbb I}}\rightarrow D_{\mathbb I}(Y)$ defined in Notation \ref{notation maps} item \textit{4.}
		It holds that $\Psi_{\mathbb I}$ is a lift of $(\psi_1,\ldots,\psi_r)\circ q_{\mathbb O_{\mathbb I_0\setminus\mathbb I}}$ and $\Psi_{\mathbb I}\circ p_{\mathbb O_{\mathbb I_0\setminus\mathbb I}}=\pi\circ (f_1,\ldots,f_n)$, hence $(\psi_1,\ldots,\psi_r)\circ q_{\mathbb O_{\mathbb I_0\setminus\mathbb I}}:\tilde X/L_{\mathbb O_{\mathbb I_0\setminus\mathbb I}} \rightarrow D_{\mathbb I_0}(Y)$ ${\mathcal P}_{\mathbb I}$-splits.

		For the second statement let  $i_s\in \mathbb I_0\setminus\mathbb I$ and $k\in \mathbb O_{i_s}$.
		Then the $\{i_s\}$ appears as a unitary set in ${\mathcal P}_{\mathbb I}$ and we can speak about the $k$-th coordinate of $D_{\mathbb I}(Y)$ and the projection to the "$k$-th component" $pr_{k}: D_{\mathbb I}(Y) \rightarrow Y$.
		Then the composition of $pr_k$ with $\Psi_{\mathbb I}$ satisfies the conditions of Lemma \ref{int cov L fi=> G<Li} and it follows that $G\subseteq L_k$.
		Since this holds for every $i_s\in  \mathbb I_0\setminus\mathbb I$ and every $k\in \mathbb O_{i_s}$, this implies that $G\subseteq L_{\mathbb O_{ \mathbb I_0\setminus\mathbb I}}$ and by \cite[page 81]{hatcher2002algebraic} this determines a composition of covering spaces $\tilde X\rightarrow \tilde X/G\rightarrow \tilde X/L_{\mathbb O_{ \mathbb I_0\setminus\mathbb I}}\rightarrow X$.	
	\end{proof}
	
	Proposition~\ref{int cov L Psi=> G<Li} induces the following (commutative) diagrams for each $n$-valued non-split map $\phi:X\multimap Y$ and ${\mathcal P}_{\mathbb I_0}$-splitting $(\psi_1,\ldots,\psi_r):X\rightarrow D_{{\mathbb I_0}}(Y)$ of the correspondent map $\Phi:X\rightarrow D_n(Y)$, where $\mathbb I_0=\{i_1,\ldots,i_r\}\subseteq\{1,\ldots,n\}$. Note that these "towers" are different for each $n$-valued non-split map $\phi:X\multimap Y$ and are therefore associated to the map $\phi$;
	
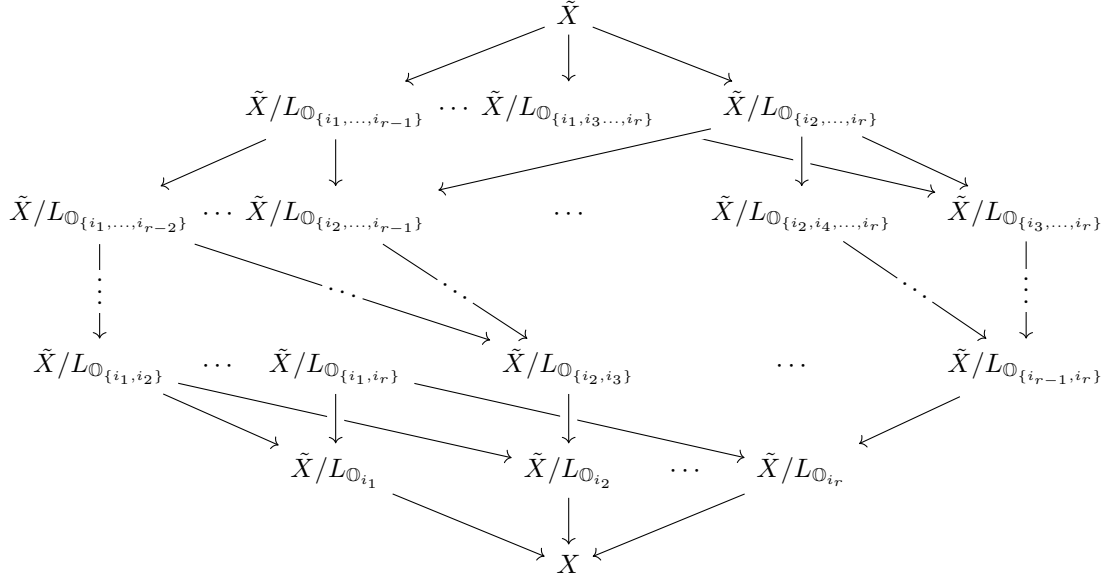
	\begin{equation}\label{D-tower}
		\begin{tikzcd}[column sep=small]
			&&\tilde X\arrow{ld}\arrow{dr}\arrow[d]&& \\
			%
			&\tilde X/L_{\mathbb O_{\{i_1,\ldots,i_{r-1}\}}}\arrow[crossing over]{ld}\arrow{d}\arrow[r,white,"\cdots"{marking, black}]
			&\tilde X/L_{\mathbb O_{\{i_1,i_3\ldots,i_{r}\}}}\arrow{drr}
			& \tilde X/L_{\mathbb O_{\{i_2,\ldots,i_{r}\}}}\arrow[d,crossing over]\arrow{dr}\arrow[crossing over]{lld}&\\
			\tilde X/L_{\mathbb O_{\{i_1,\ldots,i_{r-2}\}}}\arrow[dd,"\cdots"{marking,fill=white}]\arrow[r,white,"\cdots"{marking, black}]
			\arrow[ddrr,"\cdots"{marking,fill=white}]
			&\tilde X/L_{\mathbb O_{\{i_2,\ldots,i_{r-1}\}}}\arrow[rr,white,"\cdots"{marking, black}]
			\arrow[ddr,"\cdots"{marking,fill=white}]
			&
			& \tilde X/L_{\mathbb O_{\{i_2,i_4,\ldots,i_{r}\}}}
			\arrow[ddr,"\cdots"{marking,fill=white}]
			&\tilde X/L_{\mathbb O_{\{i_3,\ldots,i_{r}\}}}\arrow[dd,"\cdots"{marking,fill=white}]\\
			&&&&\\
			\tilde X/L_{\mathbb O_{\{i_1,i_2\}}}\arrow{dr}\arrow{drr}\arrow[r,white,"\cdots"{marking, black}]
			&\tilde X/L_{\mathbb O_{\{i_1,i_r\}}} \arrow[crossing over]{d}\arrow{drr}
			&\tilde X/L_{\mathbb O_{\{i_2,i_3\}}}\arrow[crossing over]{d}
			\arrow[rr,white,"\cdots"{marking, black}]
			& 
			&\tilde X/L_{\mathbb O_{\{i_{r-1},i_r\}}}\arrow[crossing over]{ld}\\
			&\tilde X/L_{\mathbb O_{i_1} }\arrow{dr}
			& \tilde X/L_{\mathbb O_{i_2}} \arrow{d}\arrow[r,white,"\cdots"{marking, black}]
			& \tilde X/L_{\mathbb O_{i_r}}\arrow{ld}&\\
			&&X&&
		\end{tikzcd}
	\end{equation}

	Each of these spaces can be mapped to a configuration space, such that the following diagram commutes (see Notation \ref{notation maps} for the definition of the maps);
	
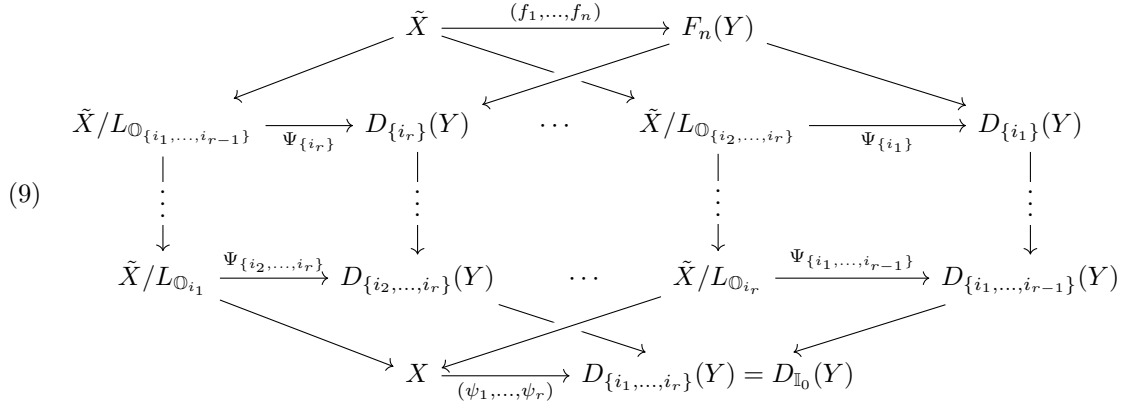
\begin{equation}\label{D-tower-D}
		\begin{tikzcd}
			&\tilde X\arrow{dr}\arrow{ld}\arrow{r}{(f_1,\ldots,f_n)}&F_n(Y)\arrow{dr}\arrow[crossing over]{ld}& \\
			\tilde X/L_{\mathbb O_{\{i_1,\ldots,i_{r-1}\}}} \arrow["\cdots"{marking,fill=white}]{dd}\arrow[swap]{r}{\Psi_{\{i_r\}}}& D_{\{i_r\}}(Y)\arrow["\cdots"{marking,fill=white}]{dd}\arrow[r,white,"\cdots"{marking, black}] & \tilde X/L_{\mathbb O_{\{i_2,\ldots,i_{r}\}}}\arrow["\cdots"{marking,fill=white}]{dd}\arrow[swap]{r}{\Psi_{\{i_1\}}}&D_{\{i_1\}}(Y)\arrow["\cdots"{marking,fill=white}]{dd}\\
			&&&\\
			\tilde X/L_{\mathbb O_{i_1}} \arrow{dr} \arrow{r}{\Psi_{\{i_2,\ldots,i_{r}\}}}& D_{\{i_2,\ldots,i_{r}\}}(Y)\arrow{dr} \arrow[white,"\cdots"{marking, black}]{r} & \tilde X/L_{\mathbb O_{i_r}}\arrow[crossing over]{ld}\arrow{r}{\Psi_{\{i_1,\ldots,i_{r-1}\}}}& D_{\{i_1,\ldots,i_{r-1}\}}(Y)\arrow{ld}\\
			&X\arrow[swap]{r}{(\psi_1,\ldots,\psi_r)}&D_{\{i_1,\ldots,i_{r}\}}(Y)=D_{\mathbb I_0}(Y)&
		\end{tikzcd}
	\end{equation}
	In Diagram \eqref{D-tower-D} it is possible to see, when the orbit of a lift factor $f_s$ "splits off" and "appears in the map".
	Note that for two different subsets $\mathbb I_1,\mathbb I_2\subseteq I_0$ of equal order with $L_{\mathbb O_{\mathbb I_0\setminus\mathbb I_1}}=L_{\mathbb O_{\mathbb I_0\setminus\mathbb I_2}}$, then $\tilde X/L_{\mathbb O_{\mathbb I_0\setminus\mathbb I_2}}=\tilde X/L_{\mathbb O_{\mathbb I_0\setminus\mathbb I_1}}=\tilde X/L_{\mathbb O_{\mathbb I_0\setminus(\mathbb I_1\cap \mathbb I_2)}}$.
	In this case we can substitute the configuration spaces and the maps $\Psi_{\mathbb I_1}$ and $\Psi_{\mathbb I_2}$ by $\Psi_{\mathbb I_1\cap \mathbb I_2}:\tilde X/L_{\mathbb O_{\mathbb I_0\setminus\mathbb I_1}}\rightarrow D_{\mathbb I_1\cap \mathbb I_2}(Y)$ and repeat this process, until there aren't two same spaces in Diagrams \eqref{D-tower} and \eqref{D-tower-D}. 
	
	Note that the partition ${\mathcal P}_{\mathbb I_1\cap\mathbb I_2}$ induced by $\mathbb I_1\cap\mathbb I_2$ is a strict refinement of both partitions ${\mathcal P}_{\mathbb I_1}$ and ${\mathcal P}_{\mathbb I_2}$.
	Thus there exists a ${\mathcal P}_{\mathbb I_1\cap\mathbb I_2}$-split of both $(\psi_1,\ldots,\psi_r)\circ q_{\mathbb O_{\mathbb I_1}}$ and $(\psi_1,\ldots,\psi_r)\circ q_{\mathbb O_{\mathbb I_2}}$.
	In general there are more possibilities for other ${\mathcal P}$-splits, where ${\mathcal P}$ is more refined.
	\begin{rem}\label{rem:refinement}
		Observe that if $L_{\mathbb O_{\mathbb I_0\setminus\mathbb I}}\subseteq \mathcal S_{t_1}\times\cdots\times \mathcal S_{t_k}$, where $(t_1,\ldots,t_k)$ is a refinement of the partition ${\mathcal P}_{\mathbb I}$, then there exists a  $(t_1,\ldots,t_k)$-split of $\Psi_{\mathbb I}$.
	\end{rem}
	
	The following example illustrates process of building a tower for a particular case.
	\begin{ex}
		In this example we explicit an $12$-valued non-split map from the torus to the disc and study the effect of the deck-transformations on the lift factors.
		Finally we build a tower like in Diagrams \eqref{D-tower} and \eqref{D-tower-D}.
		Recall that a presentation of the braid group on $n$ strands on the disc $\mathbb D$ is given as follows;
		\begin{equation*}
			B_n= \left \langle \sigma_1,\ldots,\sigma_{n-1}; \sigma_i \sigma_{i+1} \sigma_i = \sigma_{i+1} \sigma_i \sigma_{i+1}, \sigma_i \sigma_j=\sigma_j\sigma_i \text{ whenever } |i-j|>1\right \rangle.
		\end{equation*}

		Consider, in $B_{12}(\mathbb D)$ the braids $\alpha=\sigma_3\sigma_2\sigma_1$ and $\beta=\sigma_6\sigma_5\sigma_8 \sigma_{10}$ and
		the $12$-valued map from the torus to the disc $\phi:\mathbb T \stackrel{12}   \multimap  \mathbb D$, with correspondent map $\Phi:\mathbb T\rightarrow D_{12}(\mathbb D)$, defined by $\Phi_\#$ taking $a,b\in \pi_1(\mathbb T)=\{a,b:ab=ba\}$ to the braids $\alpha $ and $\beta$ respectively.
		It holds that $\alpha$ and $\beta$ commute, therefore $\Phi_\# $ is a well defined group homomorphism and since all spaces used are $K(\pi,1)$-spaces, this defines a homotopy class of maps from the torus to the disc.
		\bigskip
		
			
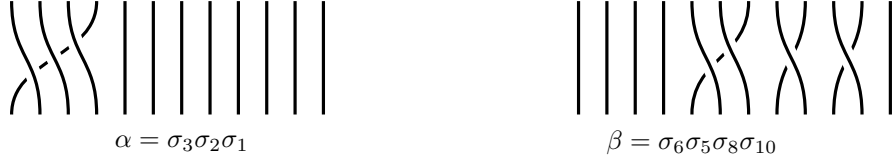
\begin{tikzpicture}[scale=0.75, very thick]
				\foreach \k in {2.5}{\draw (\k,6) .. controls (\k,5) and (\k-1.5,5) .. (\k-1.5,4);};			
				\foreach \k in {1}{\draw[white,line width=6pt](\k,6) .. controls (\k,5) and (\k+0.5,5) .. (\k+.5,4);
					\draw (\k,6) .. controls (\k,5) and (\k+0.5,5) .. (\k+0.5,4);};
				\foreach \k in {1.5}{\draw[white,line width=6pt](\k,6) .. controls (\k,5) and (\k+0.5,5) .. (\k+.5,4);
					\draw (\k,6) .. controls (\k,5) and (\k+.5,5) .. (\k+0.5,4);};
				\foreach \k in {2}{\draw[white,line width=6pt](\k,6) .. controls (\k,5) and (\k+0.5,5) .. (\k+.5,4);
					\draw (\k,6) .. controls (\k,5) and (\k+.5,5) .. (\k+.5,4);};		
				\foreach \k in {3}{\draw (\k,6)--(\k,4);};
				\foreach \k in {3.5}{\draw (\k,6)--(\k,4);};
				\foreach \k in {4}{\draw (\k,6)--(\k,4);};
				\foreach \k in {4.5}{\draw (\k,6)--(\k,4);};	
				\foreach \k in {5}{\draw (\k,6)--(\k,4);};
				\foreach \k in {5.5}{\draw (\k,6)--(\k,4);};
				\foreach \k in {6}{\draw (\k,6)--(\k,4);};
				\foreach \k in {6.5}{\draw (\k,6)--(\k,4);};	
				\foreach \k in {5}{\node at (4,3.5) {$\alpha=\sigma_3\sigma_2\sigma_1$};};
				
				\foreach \k in {11}{\draw (\k,6)--(\k,4);};
				\foreach \k in {11.5}{\draw (\k,6)--(\k,4);};
				\foreach \k in {12}{\draw (\k,6)--(\k,4);};
				\foreach \k in {12.5}{\draw (\k,6)--(\k,4);};					
				\foreach \k in {14}{\draw (\k,6) .. controls (\k,5) and (\k-1,5) .. (\k-1,4);};			
				\foreach \k in {13}{\draw[white,line width=6pt](\k,6) .. controls (\k,5) and (\k+0.5,5) .. (\k+.5,4);
					\draw (\k,6) .. controls (\k,5) and (\k+0.5,5) .. (\k+0.5,4);};
				\foreach \k in {13.5}{\draw[white,line width=6pt](\k,6) .. controls (\k,5) and (\k+0.5,5) .. (\k+.5,4);
					\draw (\k,6) .. controls (\k,5) and (\k+.5,5) .. (\k+0.5,4);};
				\foreach \k in {15}{\draw (\k,6) .. controls (\k,5) and (\k-.5,5) .. (\k-.5,4);};	
				\foreach \k in {14.5}{\draw[white,line width=6pt] (\k,6) .. controls (\k,5) and (\k+.5,5) .. (\k+.5,4);
					\draw (\k,6) .. controls (\k,5) and (\k+.5,5) .. (\k+.5,4);};
				\foreach \k in {16}{\draw (\k,6) .. controls (\k,5) and (\k-.5,5) .. (\k-.5,4);};	
				\foreach \k in {15.5}{\draw[white,line width=6pt] (\k,6) .. controls (\k,5) and (\k+.5,5) .. (\k+.5,4);
					\draw (\k,6) .. controls (\k,5) and (\k+.5,5) .. (\k+.5,4);};				
				\foreach \k in {16.5}{\draw (\k,6)--(\k,4);};				
				\foreach \k in {16}{\node at (13,3.5) {$\beta=\sigma_6\sigma_5\sigma_8 \sigma_{10}$};};
			\end{tikzpicture}

		\bigskip
		\noindent
		Consider $ker \theta=ker (\rho\circ\Phi_\#)$, which is equal to $\langle a^4,b^6\rangle$
		which induces a covering space $q:\tilde {\mathbb T}\rightarrow \mathbb T$ and a lift $\hat \Phi=\{f_1,\ldots,f_{12}\}:\tilde {\mathbb T}\rightarrow F_{12}(\mathbb D)$.
		Note that $L'=\langle (1\ 2\ 3\ 4), (5\ 6\ 7), (8\ 9)(10\ 11)\rangle \subseteq \mathcal S_{12}$.
		The group of deck-transformations $\mathcal Deck(\tilde {\mathbb T},q)$ is isomorphic to $L'$  through $\Gamma$ (see Notation \ref{notation non-split}) and we denote $\delta_1,\delta_2,\delta_3\in \mathcal Deck(\tilde {\mathbb T},q)$ as the deck-transformation such that $\Gamma(\delta_1)=(1\ 2\ 3\ 4)$,  $\Gamma(\delta_2)=(5\ 6\ 7)$, $\Gamma(\delta_3)=(8\ 9)(10\ 11)$.
		Then we have five different orbits in $L'$, namely $\mathbb O_1=\{1, 2, 3, 4\}$, $\mathbb O_5=\{5, 6, 7\}$, $\mathbb O_8=\{8, 9\}$, $\mathbb O_{10}=\{10, 11\}$ and $\mathbb O_{12}=\{12\}$, hence $\mathbb I_0=\{1, 5, 8, 10, 12\}$.
		Furthermore the stabilizer groups are $L'_{\mathbb O_1}=L'_1=\langle(5\ 6\ 7),(8\ 9)(10\ 11)\rangle$, $L'_{\mathbb O_5}=L'_5=\langle(1\ 2\ 3\ 4),(8\ 9)(10 \ 11)\rangle $ and $L'_{\mathbb O_8}=L'_8=L'_{\mathbb O_{10}}=L'_{10}=\langle (1\ 2\ 3\ 4),(5\ 6\ 7)\rangle$ and $L'_{12}=L'$.
		Also $L'_{\{1,5\}}=L'_{\mathbb O_{\{1,5\}}}=\langle (8\ 9)(10\ 11) \rangle$, $L'_{\{1,8\}}=L'_{\mathbb O_{\{1,8\}}}=L'_{\mathbb O_{\{1,10\}}}=\langle (5\ 6\ 7) \rangle$ and $L'_{\{5,8\}}=L'_{\mathbb O_{\{5,8\}}}=L'_{\mathbb O_{\{5,10\}}}=\langle (1\ 2\ 3\ 4) \rangle$.
		By Proposition \ref{Phi P-splits} there exists a ${\mathcal P}_{\mathbb I_0}$-split $\psi=(\psi_1,\psi_2,\psi_3,\psi_4, \psi_5): \mathbb T\rightarrow D_{4,3,2,2,1}(\mathbb D)$, where 
		$\psi_5=f_{12}$ and $\psi_1\circ q,\psi_2\circ q,\psi_3\circ q,\psi_4\circ q$ lift to $(f_1,f_2,f_3,f_4)$, $(f_5,f_6,f_7)$, $(f_8,f_9)$ and $(f_{10}, f_{11})$ respectively.
		Furthermore Proposition \ref{fsigmak=fdeltak} implies the following equations;
		\begin{multicols}{3}
			\begin{enumerate}[a)]
				\item $f_2=f_1\circ \delta_1$;
				\item $f_3=f_1\circ \delta_1^2$;
				\item $f_4=f_2\circ \delta_1^{3}$;	
				\item $f_1\circ \delta_2=f_1\circ \delta_3=f_1$;\\
				\item $f_6=f_5\circ \delta_2$;
				\item $f_7=f_5\circ \delta_2^{2}$;
				\item $f_9=f_8\circ \delta_3$;
				\item $f_{12}=f_{11}\circ \delta_3$;\\
				\item $f_5\circ \delta_1=f_5\circ \delta_3=f_5$;
				\item $f_8\circ \delta_1=f_8\circ \delta_2=f_8$;
				\item $f_{10}\circ \delta_1=f_{10}\circ \delta_2=f_{10}$;
				\item $f_{12}\circ \delta_1=f_{12}\circ \delta_2=f_{12}\circ \delta_3=f_{12}.$
			\end{enumerate}
		\end{multicols}

		The following diagram shows the tower of Diagram \eqref{D-tower-D}, excluding for readability's sake, the arrows between the appearing intermediate configuration spaces, which are straightforward to add, for example we have $F_{12}(\mathbb D)\rightarrow D_{\mathbf{1}^7,2,2,1}(\mathbb D)\rightarrow D_{\mathbf{1}^4,3,2,2,1}(\mathbb D) \rightarrow D_{4,3,2,2,1}(\mathbb D)$.
		Here we write $D_{\mathbf{1}^7,2,2,1}(\mathbb D)$ instead of $D_{1,1,1,1,1,1,1,2,2,1}(\mathbb D)$, and use analogous notation for the other intermediate configuration spaces, following the convention for partitions explained in the paragraphs preceding Notation~\ref{notation maps}.
		 
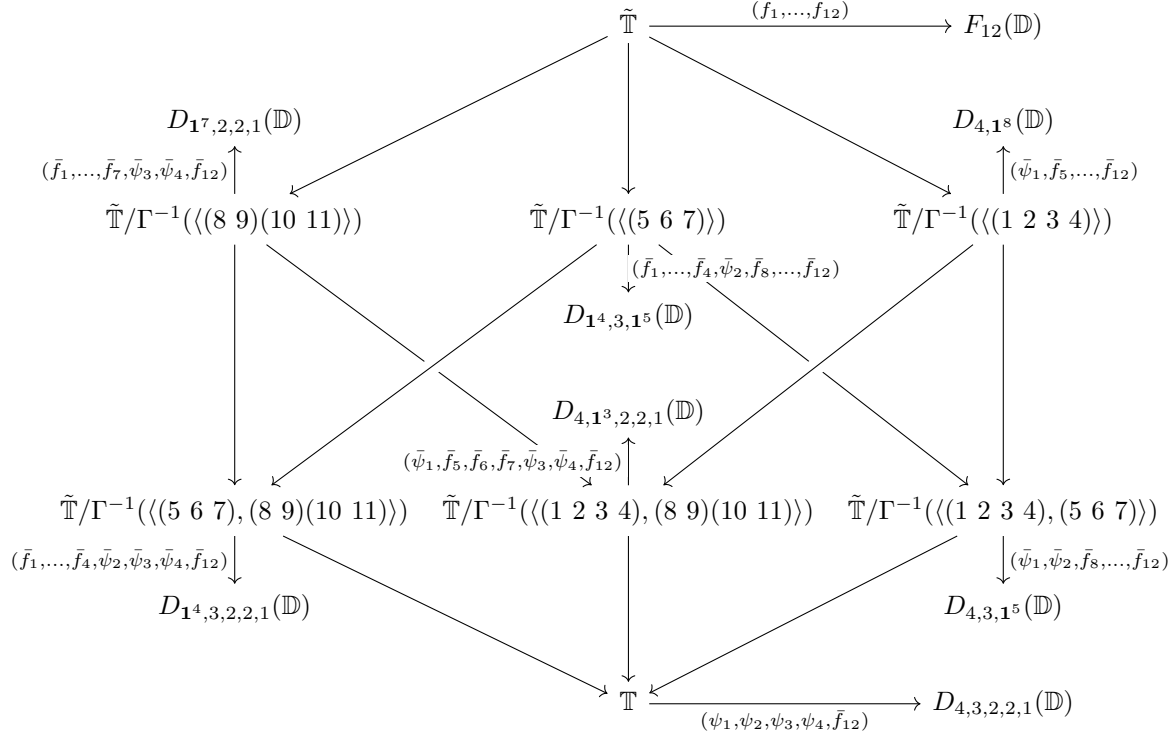
		\begin{equation*}
		\begin{tikzcd}[column sep=4pt]
			&\tilde {\mathbb T}\arrow{ldd}\arrow{dd}\arrow{rdd}\arrow{r}{(f_1,\ldots,f_{12})}
			&F_{12}(\mathbb D) \\
			D_{\mathbf{1}^7,2,2,1}(\mathbb D)&& D_{4,\mathbf{1}^8}(\mathbb D)\\
			\tilde {\mathbb T}/\Gamma^{-1}({\langle (8\ 9)(10\ 11)\rangle}) \arrow{ddd}\arrow{dddr}\arrow{u}{(\bar f_1,\ldots, \bar f_7,\bar \psi_3,\bar \psi_4,\bar f_{12})}
			&\tilde {\mathbb T}/\Gamma^{-1}({\langle (5\ 6\ 7)\rangle})\arrow{dddr}
			\arrow[d,"{(\bar f_1,\ldots, \bar f_4,\bar \psi_2,\bar f_8,\ldots,\bar f_{12})}" {fill=white, inner sep=1pt}]\arrow[crossing over]{lddd}
			&	\tilde {\mathbb T}/\Gamma^{-1}({\langle (1\ 2\ 3\ 4)\rangle})\arrow[crossing over]{lddd}\arrow[swap]{u}{(\bar \psi_1,\bar f_5,\ldots,\bar f_{12})}\arrow{ddd}
			\\
			&D_{\mathbf{1}^4,3,\mathbf{1}^5}(\mathbb D)&\\
			&D_{4,\mathbf{1}^3,2,2,1}(\mathbb D)&\\
			\tilde {\mathbb T}/\Gamma^{-1}({\langle (5\ 6\ 7),(8\ 9)(10\ 11)\rangle})
			\arrow{d}[swap]{(\bar f_1,\ldots, \bar f_4,\bar \psi_2,\bar \psi_3,\bar \psi_4,\bar f_{12})}
			\arrow{ddr}
			&\tilde {\mathbb T}/\Gamma^{-1}({\langle (1\ 2\ 3\ 4),(8\ 9)(10\ 11)\rangle})
			\arrow[u, "{(\bar \psi_1,\bar f_5,\bar f_6,\bar f_7,\bar \psi_3,\bar \psi_4,\bar f_{12})}" {fill=white, inner sep=2pt}]
			%
			\arrow{dd}
			& 	\tilde {\mathbb T}/\Gamma^{-1}({\langle (1\ 2\ 3\ 4),(5\ 6\ 7)\rangle})\arrow{d}{(\bar \psi_1,\bar \psi_2,\bar f_8,\ldots,\bar f_{12})}\arrow{ldd}
			\\
			D_{\mathbf{1}^4,3,2,2,1}(\mathbb D)&&D_{4,3,\mathbf{1}^5}(\mathbb D)\\
			& \mathbb T\arrow[swap]{r}{(\psi_1,\psi_2,\psi_3,\psi_4,\bar f_{12})}&D_{4,3,2,2,1}(\mathbb D)
		\end{tikzcd}
	\end{equation*}
	\end{ex}
	\color{black}

	\subsection{$L'_{i}$ - Stabilizer group of an integer}\label{subsection:Li}
		In this subsection we focus on stabilizer groups $L'_{i_s}$, for ${i_s}\in \mathbb I_0$, that do not stabilize the whole orbit $\mathbb O_{i_s}$.
		If $L'_{i_s}$ is normal, then $L'_{\mathbb O_{i_s}}=L'_j$, for every $j\in \mathbb O_{i_s}$, which supports the choice of the spaces in Diagrams \eqref{D-tower} and $\eqref{D-tower-D}$.	
		In the case that $L'_{i_s}$ is not normal it holds that $L'_{i_s}$ is different to $L'_{\mathbb O_{i_s}}$
		and there are possibly more than two different subgroups $L'_j$ of $L'_{\mathbb O_{i_s}}$ with $j\in \mathbb O_{i_s}.$
		Just before Notation \ref{notation action Li} we saw, that for each $j\in \mathbb O_{i_s}$, with ${i_s}\in \mathbb I$, it holds that $L'_j$ is a conjugate of $L'_{i_s}$ (and consequently isomorphic), in the following lemma we expand this statement. 
	
		\begin{lem}\label{LtauJ=tauLJtau-1}
			Let $L' \subseteq \mathcal S_n$ and let $\mathbb O_{i_s}$ 
			and $\mathbb I_0$ be as in Notation \ref{notation action Li}.
			Let furthermore ${i_s}\in \mathbb I_0$, $J\subseteq \mathbb O_{i_s}$ and $\tau \in L'$, then $L'_{\tau(J)}=\tau L'_J\tau^{-1}$, where $\tau(J)=\{\tau(j):j\in J\}$.
		\end{lem}
		\begin{proof}
			Let $\sigma\in L'_J$ and $k\in \tau(J)$, then there exists $j\in J$ such that $\tau(j)=k$ and it holds that $\sigma(j)=j$.
			Hence $(\tau \circ \sigma\circ\tau^{-1})(k)=k$ and consequently $\tau\sigma\tau^{-1}\in L'_{k}$.
			Since this holds for every $k\in \tau(J)$ it follows that $\tau L'_J\tau^{-1}\subseteq L'_{\tau(J)}$.
			For the other inclusion let $\sigma\in L'_{\tau(J)}$, then $\sigma(\tau(j))=\tau(j)$ for every $j\in J$, therefore $\tau^{-1} \circ \sigma\circ\tau\in L'_J$.
			Since $\sigma=\tau\circ(\tau^{-1} \circ \sigma\circ\tau)\circ \tau^{-1}$, it follows that $\sigma\in \tau L'_j\tau^{-1}$, for every $j\in J$ and consequently $L'_{\tau(J)}\subseteq\tau L'_J\tau^{-1}$.		
		\end{proof}
	
		The following lemma explores the natural connection between the different stabilizer groups of integers in the orbit $\mathbb O_{i_s}$ and 
		the quotient group $L'/N_{L'}(L'_{i_s})$, where $N_{L'}(L'_{i_s})$ is the normalizer of $L_{i_s}'$ in $L'$, that is $N_{L'}(L'_{i_s})=\{\tau\in L':L'_{\tau(i_s)}=L'_{i_s}\}$, using Lemma \ref{LtauJ=tauLJtau-1}.
	
		\begin{lem}\label{Oi=UtauJ}
			Let $\Phi:X\rightarrow D_n(Y)$ be an $n$-valued map inducing the structure described in Notations \ref{notation non-split} and \ref{notation action Li} and let $i_s\in \mathbb I_0$.
			Then there is a bijection between the quotient group $L'/N_{L'}(L'_{i_s})$ and $\{L'_j: j\in \mathbb O_{i_s}\}$, the stabilizer groups of all elements in the orbit of $i_s$.
		\end{lem}
		
			\begin{proof}
			Define a map taking a class $[\tau]\in L'/N_{L'}(L'_{i_s})$ to the stabilizer group $L'_{\tau(i_s)}$.
			This map is well-defined and injective, because two permutations $\tau_1,\tau_2\in L'$ are in the same  class, if and only if, $\tau_1^{-1}\circ\tau_2\in N_{L'}(L'_{i_s})$, which, using Lemma \ref{LtauJ=tauLJtau-1}, is equivalent to $L'_{\tau_1(i_s)}=L'_{\tau_2(i_s)}$.
			%
			Furthermore this map is surjective, since $\mathbb O_{i_s}=\{j\in \{1,\ldots, n\}:\exists \tau\in L' \text{ st. } \tau(i_s)=j\}$ is the orbit of $i_s$.
		\end{proof}
		Since $L'$ is finite, so is $L'/N_{L'}(L'_{i_s})$ and we denote the different classes by $[\tau_1],\ldots,[\tau_{x_s}]$, where we suppose without loss of generality that $[\tau_1]=[id]$.
		Lemma \ref{Oi=UtauJ} implies that $L'_{\tau_1(i_s)},\ldots, L'_{\tau_{x_s}(i_s)}$ are all the different stabilizer groups of elements of $\mathbb O_{i_s}$. 
		The normalizer $N_{L'}(L'_{i_s})$ also induces a partition of $\mathbb O_{i_s}$.
		For that purpose let $\mathbb J_s=\{\tau(i_s): \tau\in N_{L'}(L'_{i_s})\}$, i.e. $\mathbb J_s$ is the orbit of $i_s$ by $N_{L'}(L'_{i_s})$.
		Since $\mathbb J_s\subseteq \mathbb O_{i_s}$ is a finite set, we can write $\mathbb J_s=\{j_1,\ldots, j_{k_s}\}$, where $k_s=|\mathbb J_s|$.
		Note that $\tau_1(\mathbb J_s),\ldots,\tau_{x_s}(\mathbb J_s)$ are disjoint sets of the same order $k_s$, with $\mathbb O_{i_s}=\dot{\bigcup}_{t=1}^{x_s}\tau_t(\mathbb J_s)$.
		Furthermore Lemma \ref{LtauJ=tauLJtau-1} implies, for each $t\in\{1,\ldots,x_s\}$, that $L'_{\tau_t(\mathbb J_s)}=L'_{\tau_t(i_s)}=\tau_t L'_{i_s}\tau_t^{-1}$.
		
		Observe that in case that $L'_{i_s}$ is not normal we can expand the Diagrams \eqref{D-tower} and \eqref{D-tower-D} between $\tilde X/L_{\mathbb O_{i_s}}$ and $X$.
		The next step is to define a partition, an intermediate configuration space and a map, to be able to extend Diagram \eqref{D-tower-D}.
		Consider the partition ${\mathcal P}_{(i_s;\tau_1)}$ of $\{1,\ldots, n\}$, consisting of the different orbits by $L'$, but instead of the orbit $\mathbb O_{i_s}$ the partition contains the set $\mathbb O_{i_s}\setminus \mathbb J_s$ and the unitary sets of every element in $\mathbb J_s$, i.e. the partition is given by
		${\mathcal P}_{(i_s;\tau_1)}=\{\mathbb O_{i_1},\ldots, \mathbb O_{i_{s-1}}, \{j_1\},\ldots, \{j_{k_s}\},\mathbb O_{i_s}\setminus \mathbb J_s,\mathbb O_{i_{s+1}},\ldots, \mathbb O_{i_{r}}\}.$
		This partition induces the intermediate configuration space $D_{{\mathcal P}_{(i_s;\tau_1)}}(Y)$, which we also denote by $D_{(i_s;\tau_1)}(Y)$ (see the paragraph proceeding Definition \ref{def:P-split} to recall the denotation $D_{{\mathcal P}_{(i_s;\tau_1)}}(Y)$).
		We define a map from $\tilde X/L'_{i_s}$ to $D_{{\mathcal P}_{(i_s;\tau_1)}}(Y)$ being $\Psi_{(i_s;\tau_1)}=(\bar\psi_1,\ldots,\bar\psi_{s-1},\bar f_{j_1}\,\ldots,\bar f_{j_{k_s}},\{\bar f_j:j\in \mathbb O_{i_s}\setminus \mathbb J_s\},\bar\psi_{s+1},\ldots, \bar\psi_r)$.
		Analogously we can define for any $t\in \{1,\ldots,x_s\}$ the partition ${\mathcal P}_{(i_s;\tau_1)}=\{\mathbb O_{i_1},\ldots, \mathbb O_{i_{s-1}}, \{\tau_t(j_1)\},\ldots, \{\tau_t(j_{k_s})\},\mathbb O_{i_s}\setminus \mathbb \tau_t(J_s),\mathbb O_{i_{s+1}},\ldots, \mathbb O_{i_{r}}\}$ and the equivalent intermediate configuration space $D_{(i_s;\tau_t)}(Y)=D_{{\mathcal P}_{(i_s;\tau_t)}}(Y)$.
		Furthermore we define a map $(\bar\psi_1,\ldots,\bar\psi_{s-1},\bar f_{\tau_t(j_1)}\,\ldots,\bar f_{\tau_t(j_{k_s})},\{\bar f_j:j\in \mathbb O_{i_s}\setminus \mathbb \tau_t(J_s)\},\bar\psi_{s+1},\ldots, \bar\psi_r)$ going from $\tilde X/L_{\tau_t(i_s)}$ to $D_{(i_s;\tau_t)}(Y)$, which we denote by $\Psi_{(i_s;\tau_t)}$.
		Note that $\tilde X/L_{i_s}$ and $D_{(i_s;\tau_1)}(Y)=D_{(i_s;id)}(Y)$ are isomorphic to $\tilde X/L_{\tau_t(i_s)}$ and $D_{(i_s;\tau_t)}(Y)$, respectively and $\Psi_{(i_s;\tau_t)}$ can be obtained analogously to $\Psi_{(i_s;\tau_1)}=\Psi_{(i_s;id)}$, hence we consider $t=1$ from here on out.
		We have the following commutative diagram, extending the path of the orbit $\mathbb O_{i_s}$ of the tower in Diagram \eqref{D-tower-D};
		\begin{equation}\label{D-diagram-not normal}
			\begin{tikzcd}[column sep=huge]
				\tilde X\arrow["\cdots"{marking,fill=white}]{dd}\arrow{r}{(f_1,\ldots,f_n)}
				&
				F_{n}(Y)
				\arrow["\cdots"{marking,fill=white}]{dd}{\pi}\\
				&\\
				\tilde X/L_{\mathbb O_{i_s}}\arrow{d}\arrow{r}{\Psi_{\{i_s\}}}
				& D_{\{i_s\}}(Y)\arrow{d}{\pi}\\
				\tilde X/L_{i_s} \arrow{d}\arrow{r}{\Psi_{(i_s;id)}}
				& D_{(i_s;id)}(Y)
				\arrow{d}{\pi}\\
				X\arrow{r}{(\psi_1,\ldots,\psi_{r})}&D_{\mathbb I_0}(Y)
			\end{tikzcd}
		\end{equation}
	
		Note that this diagram could be expanded in multiple ways;
		We could intersect $L'_{i_s}$ with every combination of $L'_{\tau_t(i_s)}$'s, where $t\in \{1,\ldots,x_s\}$ and consider the spaces $\tilde X/L_{\{i_s,\tau_{t_1}(i_s),\ldots, \tau_{t_v}(i_s)\}}$, where the maps are analogously defined.
		As an example consider the space $\tilde X/L_{\{i_s,\tau_{t}(i_s)\}}$, then 
		$\Psi_{(i_s;id,\tau_{t})}=(\bar\psi_1,\ldots,\bar\psi_{s-1},\bar f_{j_1}\,\ldots,\bar f_{j_{k_s}},\bar f_{\tau_{t}(j_1)}\,\ldots,\bar f_{\tau_{t}(j_{k_s})},\{\bar f_j:j\in \mathbb O_{i_s}\setminus(\mathbb J_s\cup\tau_t(\mathbb J_s))\},\bar\psi_{s+1},\ldots, \bar\psi_r)$ is the map between $\tilde X/L_{\{i_s,\tau_{t}(i_s)\}}$ and $D_{(i_s;id,\tau_{t})}(Y)$, defined analogously to $\Psi_{(i_s;id)}$.
		We could also intersect $L'_{i_s}$ with stabilizer groups of another orbit $L'_{\mathbb O_{i_{\hat s}}}$ and, in case that $L'_{{i_{\hat s}}}$ is not normal, with the different stabilizers of the orbit of $i_{\hat s}$.
		In this case we would change in $\Psi_{(i_s;id,\tau_{t})}$ the map $\bar \psi_{\hat s}$ accordingly.
		
		Another aspect that could be explored more is to "refine" the lifts for each of these spaces.
		For example to find for $\tilde X/L_{i_s}$ the finest partition ${\mathcal P}$ of $\{1,\ldots,n\}$ such that $(\psi_1,\ldots, \psi_r)\circ  q_{i_s}:\tilde X/L_{i_s}\rightarrow D_{\mathbb I_0}(Y)$ can be lifted to $D_{{\mathcal P}}(Y)$.
		In the following example we construct a tower like in Diagram \eqref{D-diagram-not normal}, and observe that in this case we can refine the intermediate configuration space.

		\begin{ex}
			We give an example of a $6$-valued map from the Klein-bottle $\mathbb K$ to the sphere $\mathbb S$, where the stabilizer groups are not normal in $L'$.
			We use the following presentation of the braid group on $6$ strands on the sphere, with generators $\sigma_1,\ldots, \sigma_{n-1}$ and following relations;
			\begin{enumerate}[1.]
				\item $\sigma_i \sigma_j = \sigma_j \sigma_i$, for $|i-j| \geq 2 $,
				\item $ \sigma_i \sigma_{i+1} \sigma_i = \sigma_{i+1} \sigma_i \sigma_{i+1}$ for $1 \leq i \leq n - 2$ (the Yang–Baxter equation),
				\item $\sigma_{n-1}  \cdots \sigma_1^2 \cdots \sigma_{n-1} = 1$.
			\end{enumerate}
			Consider the $6$-valued non-split map $\phi:\mathbb K\rightarrow \mathbb S^2$ from the Klein bottle with $\pi_1(\mathbb K)=\langle \alpha,\beta: \alpha\beta\alpha=\beta^{-1}\rangle$ to the sphere, where $\Phi_\#:\pi_1(\mathbb K)\rightarrow B_4(\mathbb S^2)$, the induced homomorphism of the correspondent map, takes $\alpha$ to $\hat \alpha=\sigma_5\sigma_4\sigma_3\sigma_2\sigma_1$ and $\beta$ to
			$\hat \beta=\sigma_1\sigma_2\sigma_3\sigma_4 \sigma_1 \sigma_2\sigma_3 \sigma_1 \sigma_2\sigma_1$.\\
			                              
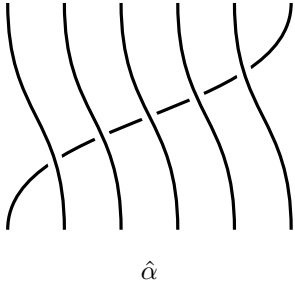
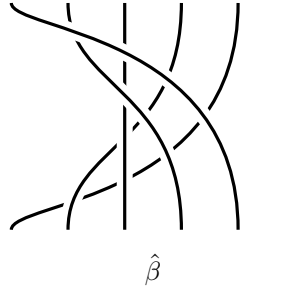
			\begin{tikzpicture}[scale=0.75, very thick]
				\foreach \k in {6}{\draw (\k,6) .. controls (\k,4) and (\k-5,4) .. (\k-5,2);};
				\foreach \k in {1}{\draw[white,line width=6pt] (\k,6) .. controls (\k,4) and (\k+1,4) .. (\k+1,2);
					\draw (\k,6) .. controls (\k,4) and (\k+1,4) .. (\k+1,2);};
				\foreach \k in {2}{\draw[white,line width=6pt] (\k,6) .. controls (\k,4) and (\k+1,4) .. (\k+1,2);
					\draw (\k,6) .. controls (\k,4) and (\k+1,4) .. (\k+1,2);};
				\foreach \k in {3}{\draw[white,line width=6pt] (\k,6) .. controls (\k,4) and (\k+1,4) .. (\k+1,2);
					\draw (\k,6) .. controls (\k,4) and (\k+1,4) .. (\k+1,2);};
				\foreach \k in {4}{\draw[white,line width=6pt] (\k,6) .. controls (\k,4) and (\k+1,4) .. (\k+1,2);
					\draw (\k,6) .. controls (\k,4) and (\k+1,4) .. (\k+1,2);};
				\foreach \k in {5}{\draw[white,line width=6pt] (\k,6) .. controls (\k,4) and (\k+1,4) .. (\k+1,2);
					\draw (\k,6) .. controls (\k,4) and (\k+1,4) .. (\k+1,2);};
				\foreach \k in {5}{\node at (3.5,1.3) {$\hat\alpha$};};
				
				\foreach \k in {16}{\draw (\k,6) .. controls (\k,2.5) and (\k-4,2.5) .. (\k-4,2);};			
				\foreach \k in {15}{\draw[white,line width=6pt] (\k,6) .. controls (\k,3.5) and (\k-2,3.5) .. (\k-2,2);
					\draw (\k,6) .. controls (\k,3.5) and (\k-2,3.5) .. (\k-2,2);};
				\foreach \k in {14}{\draw[white,line width=6pt] (\k,4) -- (\k,2);
					\draw (\k,6) -- (\k,2);};
				\foreach \k in {13}{\draw[white,line width=6pt] (\k,6) .. controls (\k,4.5) and (\k+2,4.5) .. (\k+2,2);
					\draw (\k,6) .. controls (\k,4.5) and (\k+2,4.5) .. (\k+2,2);};
				\foreach \k in {12}{\draw[white,line width=6pt] (\k,6) .. controls (\k,5.5) and (\k+4,5.5) .. (\k+4,2);
					\draw (\k,6) .. controls (\k,5.5) and (\k+4,5.5) .. (\k+4,2);};
				\foreach \k in {17}{\draw (\k,6) -- (\k,2);};	
				\foreach \k in {15}{\node at (14.5,1.3) {$\hat\beta$};};\\
			\end{tikzpicture}
			
			These two braids anti-commute, therefore $\Phi:\mathbb K\rightarrow D_4(\mathbb S^2)$ is well-defined.
			It holds that $\rho(\hat\alpha)=(123456)$ and $\rho(\hat\beta)=(15)(24)$, where $\rho:B_6(\mathbb S^2)\rightarrow \mathcal S_6$ takes the braid to its underlying permutation, as in Notation \ref{notation non-split}.
			Consequently $L'=im (\rho\circ \Phi_\#)=\langle (123456), (15)(24)\rangle=\langle \tau_1,\tau_2: \tau_1^6=id, \tau_2^2=id, \tau_1\tau_2\tau_1=\tau_2\rangle$.
			It follows that $\mathbb O_1=\{1,2,3,4,5,6\}$ and $L'_1=L'_4=\langle(26)(35)\rangle$, $L'_2=L'_5=\langle(13)(46)\rangle$ and $L'_3=L'_6=\langle(15)(24)\rangle$ are non-normal subgroup of $L'$.
			
			Consider the covering space $q:\tilde {\mathbb T}\rightarrow \mathbb K$ induced by $ker \theta=ker(\Phi_\#\circ \rho)$, which is isomorphic to $\pi_1(\tilde {\mathbb T}) = \langle a,b:ab=ba\rangle$, where we identify $a$ with $\alpha^6$ and $b$ with $\beta^2$.
			Let furthermore $\delta_1=\Gamma^{-1}((123456)), \delta_2=\Gamma^{-1}((15)(24))$, which generate $\mathcal Deck(\tilde {\mathbb T},q)$, since $\Gamma:\mathcal Deck(\tilde {\mathbb T},q)\rightarrow L'$ is an isomorphism (see Notation \ref{notation non-split}).
			Let $\hat \Phi=(f_1,f_2,f_3,f_4,f_5,f_6):\tilde {\mathbb T}\rightarrow F_6(\mathbb S^2)$ be the lift of $\Phi\circ q:\tilde {\mathbb T}\rightarrow D_6(\mathbb S)$.			
			%
			The tower of Diagram \eqref{D-diagram-not normal} in this case is the following;
			 
			\begin{equation*}
				\begin{tikzcd}[column sep=100pt]
					\tilde {\mathbb T}\arrow{d} \arrow{r}{(f_1,\ldots,f_6)}
					&F_{6}(\mathbb S)\arrow{d}{\pi}\\
					\tilde{\mathbb T}/\Gamma^{-1}(\langle (2\ 6)(3\ 5)\rangle)\arrow{d}\arrow{r}{(\bar f_1, \{\bar f_2, \bar f_3, \bar f_5, \bar f_6\}, \bar f_4)}
					&D_{\{1\},\{2,3,5,6\},\{4\}}(\mathbb S)\arrow{d}{\pi}\\
					\mathbb K\arrow[swap]{r}{\Phi}&D_{6}(\mathbb S)
				\end{tikzcd}
			\end{equation*}
					
			Consider $D_{\{1\},\{2,6\},\{3,5\},\{4\}}(\mathbb S)=F_6(\mathbb S)/\langle (2\ 6),(3\ 5)\rangle\subseteq F_6(\mathbb S)/\langle (2\ 6)(3\ 5)\rangle$, then there is a lift from $\tilde{\mathbb T}/\Gamma^{-1}(\langle (2\ 6)(3\ 5)\rangle)$ to the refined intermediate configuration space $D_{\{1\},\{2,6\},\{3,5\},\{4\}}(\mathbb S)$ given by $(\bar f_1,\{\bar f_2,\bar f_6\},\{\bar f_3,\bar f_5\},\bar f_4)$, giving rise to the following (commutative) diagram;
			
			\begin{equation*}
				\begin{tikzcd}[column sep=100pt]
					\tilde {\mathbb T}\arrow{dd} \arrow{r}{(f_1,\ldots,f_6)}	&F_{6}(\mathbb S)\arrow{d}{\pi}\\
					&D_{\{1\},\{2,6\},\{3,5\},\{4\}}(\mathbb S)\arrow{d}{\pi}\\
					\tilde{\mathbb T}/\Gamma^{-1}(\langle (2\ 6)(3\ 5)\rangle)\arrow{d}\arrow[dashed]{ru}{(\bar f_1, \{\bar f_2, \bar f_6\},\{\bar f_3, \bar f_5\}, \bar f_4)}\arrow{r}[swap]{(\bar f_1, \{\bar f_2, \bar f_3, \bar f_5, \bar f_6\}, \bar f_4)}
					&D_{\{1\},\{2,3,5,6\},\{4\}}(\mathbb S)\arrow{d}{\pi}\\
					\mathbb K\arrow[swap]{r}{\Phi}&D_{6}(\mathbb S)
				\end{tikzcd}
			\end{equation*}		
			
			The following diagram follows the paths through all stabilizer groups (which are isomorphic to each other), displaying the lifts to the already refined intermediate configuration spaces, but omitting the arrows between the intermediate configuration spaces, which make the diagram commutative; 
			\begin{equation*}
				\begin{tikzcd}[column sep=small]
					&& \tilde {\mathbb T}\arrow{lldd}\arrow{ldd}\arrow{rdd}\arrow{r}{(f_1,\ldots,f_6)}
					&F_{6}(\mathbb S)\\
					&&&\\
					\tilde{\mathbb T}/\Gamma^{-1}(\langle (2\ 6)(3\ 5)\rangle)\arrow{ddrr}\arrow{dd}{(\bar f_1,\{\bar f_2,\bar f_6\},\{\bar f_3,\bar f_5\}, \bar f_4)}
					&	\tilde {\mathbb T}/\Gamma^{-1}(\langle (1\ 3)(4\ 6)\rangle)\arrow{ddr}\arrow{dr}{( \{\bar f_1,\bar f_3\},\bar f_2,\{\bar f_4,\bar f_6\},\bar f_5)}
					&& \tilde {\mathbb T}/\Gamma^{-1}(\langle (1\ 5)(2\ 4)\rangle)\arrow{ldd}\arrow{d}{( \{\bar f_1,\bar f_5\},\{\bar f_2,\bar f_4\},\bar f_3, \bar f_6)}&\\
					&&D_{\{1,3\},\{2\},\{4,6\},\{5\}}(\mathbb S)&D_{\{1,5\},\{2,4\},\{3\},\{6\}}(\mathbb S)\\
					D_{\{1\},\{2,6\},\{3,5\},\{4\}}(\mathbb S)&& \mathbb K\arrow[swap]{r}{\Phi}&D_{6}(\mathbb S)
				\end{tikzcd}
			\end{equation*}	
		\end{ex}
				\color{black}

	\subsection{Regular groups and lift factors}\label{subsection: regular groups}
		In this subsection we recall selected concepts from group theory and analyze their effect on the structure of the tower of Diagrams \eqref{D-tower-D} and \eqref{D-diagram-not normal}.
		For that purpose let $G$ be a subgroup of $\mathcal S_n$ acting on $\{1,\ldots,n\}$.
		Then $G$ \textit{acts transitively on $\{1,\cdots,n\}$}, if for every $i,j\in \{1,\ldots,n\}$ there is a $\tau\in G$ such that $\tau(i)=j$,
		$G$ \textit{acts without fixed points} or is \textit{semi-regular}, if $\tau(i) = i$ for $\tau\in G$ and some $i\in \{1,\cdots,n\}$ implies that $\tau$ is the identity.
		We call $G$ \textit{regular} if $G$ acts both transitively and without fixed points.
		It is straightforward to show that if $G$ satisfies two of the following condition, then $G$ is regular;
		\begin{enumerate}
			\item $G$ acts transitively on $\{1,\cdots,n\}$, i.e. for every $i,j\in \{1,\ldots,n\}$ there is a $\tau\in G$ such that $\tau(i)=j$;
			\item $G$ acts without fixed points, i.e. $\tau(i) = i$ for $\tau\in G$ and some $i\in \{1,\cdots,n\}$, if and only if, $\tau$ is the identity permutation;  
			\item $|G|=n$.
		\end{enumerate}		
	
	Back to our context of Notations \ref{notation non-split}, \ref{notation action Li} and \ref{notation maps} and Diagrams \eqref{D-tower-D} and \eqref{D-diagram-not normal}, note that if $L'$ is semi-regular, then it holds that $L'_i=\{\tau\in L': \tau(i)=i\}=\{id\}$, for every $i\in \{1,\ldots,n\}$ and consequently $L_i=\{id\}$.
	Therefore both Diagrams \eqref{D-tower-D} and \eqref{D-diagram-not normal} just consist of the first and last lines, i.e. the diagrams will both be as in Diagram \eqref{split semi-regular};
	
	\begin{multicols}{2}
		\noindent
		\begin{equation}\label{split semi-regular}
			\begin{tikzcd}[column sep=huge]
				\tilde X\arrow{d}\arrow{r}{(f_1,\ldots,f_n)}&F_n(Y)\arrow{d}\\
				X\arrow[swap]{r}{\Phi}&D_{\mathbb I_0}(Y)
			\end{tikzcd}
		\end{equation}
		
		\noindent
		\begin{equation}\label{eq:classiclift_nonsplit2}
			\begin{tikzcd}[column sep=huge]
				\tilde X\arrow{d}\arrow{r}{(f_1,\ldots,f_n)}&F_n(Y)\arrow{d}\\
				X\arrow[swap]{r}{\Phi}&D_n(Y)
			\end{tikzcd}
		\end{equation}
	\end{multicols}
	
	If $L'$ is transitive, then $\mathbb O_1=\{1,\ldots,n\}$, i.e. $\mathbb I_0=\{1\}$, thus the covering $D_{m_1,\ldots,m_r}(Y)\rightarrow D_{n}(Y)$ (and consequently the ${\mathcal P}_{\mathbb I_0}$-split) is trivial.
	Hence Diagram \eqref{D-tower-D} also becomes Diagram \eqref{split semi-regular}, but with $D_{\mathbb I_0}(Y)=D_n(Y)$, i.e. the same as the "trivial" Diagram \eqref{eq:classiclift_nonsplit2}.
	But even with only one orbit it is possible that some single-valued maps $f_i$ "split" from the map $\Phi$.
	If, for an  $i\in \{1,\ldots,n\}$ it happens that $L'_i$ is not normal in $L'$, then Diagram \eqref{D-diagram-not normal} contains the intermediate covering space $\tilde X/L_i$, which is neither equal to $\tilde X$ nor to $X$, with a lift to a configuration space $D_P(Y)$, where $\{i\}$ is a set of the partition ${\mathcal P}$.
	
	If $L'$ is regular, i.e. both semi-regular and acting transitively, then Diagrams \eqref{D-tower-D} and \eqref{D-diagram-not normal} "shrink"  to
	the "trivial" Diagram \eqref{eq:classiclift_nonsplit2}.
	This means there is no  intermediate covering space, where we can "separate" any single-valued maps from $\Phi$.
	
\section{Borsuk-Ulam property}\label{sec:BU}
	In this section, we analyze the relation between the lift factors of a non-split $n$-valued map $\phi:X\multimap Y$ and the Borsuk-Ulam property.
	Let $\tilde X$ and $Y$ be topological spaces and let $\delta:\tilde X\rightarrow \tilde X$ be a free involution, i.e. $\delta$ is fixed point free and $\delta^2=Id_{\tilde X}$.
	We say that a homotopy class $[f_0]\in [\tilde X,Y]$ has the \textit{Borsuk-Ulam property with respect to $\delta$} if, for every $f\in [f_0]$, there exists an $\tilde x\in \tilde X$ such that $f(\tilde x)=(f\circ\delta)(\tilde x)$.
	Observe that if $[f]\in [\tilde X,Y]$ is a homotopy class that does not have the Borsuk-Ulam property in respect to a free involution $\delta:\tilde X\rightarrow \tilde X$,
	then there is a map $f_1\in [f]$ such that $f_1$ and $f_1\circ \delta$  are coincidence  free.  
	Consider the map $\hat \Phi=(f_1,f_1\circ\delta):\tilde X\rightarrow F_2(Y)$ and the quotient space $X=\tilde X/\delta$.
	Since $\hat \Phi$ is $\mathbb{Z}_2$-equivariant, it induces a map on the quotient spaces $\Phi:X\rightarrow D_2(Y)$, which has a correspondent $2$-valued non-split map $\phi:X\multimap Y$.
	Conversely, let $\phi:X\multimap Y$ be a $2$-valued non-split map, $\Phi:X\rightarrow D_2(Y)$ its correspondent map and $\hat\Phi=(f_1,f_2):\tilde X\rightarrow F_2(X)$ the lift of $\Phi\circ q$, as in Notation \ref{notation non-split}.
	Then $q:\tilde X\rightarrow X$ is a $2$-sheeted covering space and there is a unique non-trivial deck-transformation $\delta\in \mathcal Deck(\tilde X,q)$.
	It holds that $\Gamma$ takes $\delta$ to $\tau=(12)$, the only non-trivial deck-transformation in $\mathcal Deck(F_2,\pi)$ and
	Lemma \ref{fsigmak=fdeltak} implies that $f_2=f_{\tau(1)}=f_1\circ\delta$.
	Since $\hat \Phi=(f_1,f_2)$ is a map to $F_2(Y)$, $f_1$ and $f_1\circ \delta=f_2$ are coincidence free.
	Therefore $[f_1]$ does not have the Borsuk-Ulam property in respect to $\delta$ (for more details see \cite{maues2025computation}). 
	
	It is possible to generalize the Borsuk-Ulam property to a finite group action, instead of an involution (see \cite{gonccalves2022free}).
	For that purpose, let $\tilde X$ and $Y$ be topological spaces and let $G$ be a finite non-trivial group with free action $G\times \tilde X\rightarrow\tilde X$. 
	Then every $g\in G$ induces a homeomorphism on $\tilde X$, which we denote, abusing notation, as $g:\tilde X\rightarrow \tilde X$.
	We say that a homotopy class $[f_0]\in [\tilde  X, Y]$ has the \textit{Borsuk-Ulam property with respect to $G$} if, for every $f\in [f_0]$ there exist distinct $g_1, g_2\in G$ and an $\tilde x\in \tilde X$ such that $f(g_1(\tilde x))=f(g_2(\tilde x))$.

	Inspired by the case for $n=2$ it is possible to establish an $n$-valued non-split map using a homotopy class that does not satisfy the Borsuk-Ulam property in respect to $G$, where  $G$ is a group of cardinality  $n$.
	The result is known (see \cite[Theorem 2.4.]{gonccalves2022free} and \cite[Theorem 6]{gonccalves2024borsuk}), but we formulate it again in the next proposition in a form which is more suitable for us in this paper.
	
		\begin{prop}\label{NonBUmult}
			Let $G$ with order $|G|=n$ be a non-trivial finite group with free action $G\times \tilde X\rightarrow\tilde X$ on a space $\tilde X$.
			Let  furthermore $\delta_1,\ldots,\delta_{n}$ be the $n$ distinct elements of $G$, where $\delta_1=id$, and let $[f_0]\in [\tilde X, Y]$ be a homotopy class, that does not satisfy the Borsuk-Ulam property in respect to $G$.
			Then there is an $n$-valued non-split map $\Phi:\tilde X/G\rightarrow D_n(Y)$, where $L'$, given as in Notation \ref{notation non-split}, is isomorphic to $G$.
			Furthermore, $L'$ is the regular subgroup of $\mathcal S_n$ whose elements are given by right multiplication on $G$ (as in the proof of Cayley's theorem). 
		\end{prop}
		\begin{proof}
			From the hypothesis, there exists a map $f_1\in [f_0]$ such that $f_1\circ \delta_1,\ldots f_1\circ \delta_{n}$ are pairwise coincidence-free.
			So we can define the following map $\hat \Phi=(f_1\circ \delta_1,\ldots, f_1\circ \delta_{n}):\tilde X\rightarrow F_n(Y)$. 
			By \cite[Lemma 2]{gonccalves2024borsuk} the map $\hat\Phi$ is equivariant, as is described in the next few sentences (for more details see the proof of \cite[Lemma 2]{gonccalves2024borsuk}).
				For every $\delta\in G$ consider $(f_1\circ \delta_1\circ \delta,\ldots, f_1\circ \delta_{n}\circ \delta)$.
				Since $G=\{\delta_1,\ldots,\delta_n\}$, it holds that $\delta_1\circ \delta=\delta_j$, where $j\in \{1,\ldots,n\}$.
				Analogously for any $i\in \{1,\ldots, n\}$, there exists a $j\in \{1,\ldots, n\}$, such that $\delta_i\circ \delta=\delta_j$.
				Hence $\delta$ induces a permutation $\tau$, where $\tau(i)=j$, which defines an isomorphism $\Gamma'$ between $G$ and a subgroup $G'$ of $\mathcal S_n$, taking every $\delta$ to its induced permutation $\tau$.		
			Then it holds, for every $\delta$ and $\tau=\Gamma'(\delta)$, that
			\begin{equation*}
				(f_1\circ \delta_1\circ \delta,\ldots, f_1\circ \delta_{n}\circ \delta)=(f_1\circ \delta_{\tau(1)},\ldots, f_1\circ \delta_{\tau(n)}),
			\end{equation*}
			that is $\hat \Phi \circ \delta = \Gamma'(\delta)\cdot\hat \Phi$.
			Hence the elements of $G'$ are given by right multiplication on $G$ and it is straightforward to show that $G'$ is regular.
			Furthermore the map $\hat \Phi$ is equivariant, hence it induces a map on the quotient spaces $\Phi:\tilde X/G\rightarrow D_n(Y)$, which, by construction, is non-split, $G=\mathcal Deck(\tilde X,q)$ (see \cite[Proposition 1.39.]{hatcher2002algebraic}) and $L'=G'$.
		\end{proof}
		
	Note that in the construction of $\Phi$ in the proof of Proposition \ref{NonBUmult} we obtain that the lift of $\Phi\circ q$ is given by $\hat \Phi=(f_1\circ \delta_1,\ldots, f_1\circ \delta_{n})$, where $f_1\in [f_0]$ and $\delta_1,\ldots,\delta_n\in \mathcal Deck(\tilde X,q)$, since by $\mathcal Deck(\tilde X,q)=G$.
	We would wish that the converse of Proposition \ref{NonBUmult} follows, as it did for the case where $n=2$.
	Hence we would wish that the following two facts hold, first that for every $n$-valued non-split map $\Phi:X\rightarrow Y$ the 
	homotopy class of the lift factor $f_1$ does not satisfy the Borsuk-Ulam property in respect to the action given by $ \mathcal Deck(\tilde X,q)$.
	Secondly that, for each $j=1,\ldots n$ there exists $\delta\in \mathcal Deck(\tilde X,q)$, such that $f_j=f_1\circ \delta$.
	Unfortunately either of these two statements don't hold without further restrictions.
	
	To illustrate the complexity of a converse formulation, which will be useful, we explicit two examples of $4$-valued non-split maps, where we study the effect of the deck-transformations on the lift factors $f_1,f_2,f_3$ and $f_4$.
	These examples also suggest how to formulate the converse questions.
	
	\begin{ex}\label{ex Z2+Z2}
		In this example we display two $4$-valued non-split maps from the torus $\mathbb T$ to the disc $\mathbb D$, such that $L'$ in both cases is isomorphic to $\mathbb Z_2\oplus\mathbb Z_2$, but where their lift factors behave differently in respect to the group action $\mathcal Deck( \mathbb {\tilde T},q)$.
		See Notation \ref{notation non-split} to recall the meaning of $L'$, $\mathcal Deck(\tilde X,q)$, $\Gamma$ etc.
		The following commutative diagrams illustrate both examples, where $i=1,2$;	
			\begin{multicols}{2}
					\noindent
					\begin{equation*}
						\begin{tikzcd}[column sep=huge]
							\pi_1(\mathbb T)\arrow{r}{(\Phi_i)_\#}\arrow{dr}[swap]{\theta_i}&B_4(\mathbb D)\arrow{d}{\rho}\\
							&\mathcal S_4.
						\end{tikzcd}
					\end{equation*}	
								\noindent
					\begin{equation*}
						\begin{tikzcd}[column sep=huge]
							\tilde {\mathbb T} \arrow{d}[swap]{q_i}\arrow{r}{
								\{f_1,f_2,f_3,f_4\}} &F_4(\mathbb D)\arrow{d}{\pi}\\
							\mathbb T\arrow{r}[swap]{\Phi_i}&D_4(\mathbb D).
						\end{tikzcd}
					\end{equation*}
				\end{multicols}
				\noindent
	On that account, recall that a presentation of the braid group on $n$ strands on the disc $\mathbb D$ is given as follows;
		\begin{equation*}
			B_n= \left \langle \sigma_1,\ldots,\sigma_{n-1}; \sigma_i \sigma_{i+1} \sigma_i = \sigma_{i+1} \sigma_i \sigma_{i+1}, \sigma_i \sigma_j=\sigma_j\sigma_i \text{ for } |i-j|>1 \right \rangle.
		\end{equation*}	
		
		\begin{enumerate}[a)]
			\item \label{ex Z2+Z2 not reg}
				Consider the map $\Phi_1:\mathbb T\rightarrow D_4(\mathbb D)$ from the torus to the disc such that the correspondent homomorphism $(\Phi_1)_\#$ takes $a,b$, the generators of $\pi_1(\mathbb T)$, to $\sigma_1$, respectively to $\sigma_3$ on $B_4$, braids on the disc.
				These two braids commute, therefore such map $\Phi$ does exist.
				Then $\rho(\sigma_1)=(12)$ and $\rho(\sigma_3)=(34)$ in $\mathcal S_4$.
				Consider $ker \theta_1=\langle a^2, b^2 \rangle$, which induces a covering space $q_1:\tilde {\mathbb T}\rightarrow \mathbb T$ and a lift $\hat \Phi_1=\{f_1,f_2,f_3,f_4\}:\tilde {\mathbb T}\rightarrow F_4(\mathbb D)$;
							The group of deck-transformations $\mathcal Deck(\tilde {\mathbb T}, q_1)$ is isomorphic, through $\Gamma_1$, to $L'_1=im\theta=im (\rho \circ\Phi_\#)=\langle (12),(34)\rangle\simeq \mathbb Z_2\oplus \mathbb Z_2$.
				We denote $\delta_1=\Gamma_1^{-1}((12))$ and $\delta_2=\Gamma_1^{-1}((34))$, which generate $ \mathcal Deck(\tilde  {\mathbb T},q_1)$.
				Proposition \ref{fsigmak=fdeltak} implies that $ f_2=f_1\circ \delta_1$.
				Since $f_1$ and $f_2$ are pairwise coincidence-free, so are $f_1$ and $f_1\circ \delta_1$, therefore $f_1$ does not satisfy the Borsuk-Ulam property with respect to $\langle \delta_1\rangle$.
				Analogously $f_4=f_3\circ \delta_2.$ and $f_3$ does not satisfy the Borsuk-Ulam property with respect to $\langle \delta_2\rangle$.
				But also from Proposition \ref{fsigmak=fdeltak} it holds that $f_1=f_1\circ \delta_2$ and $f_3=f_3\circ \delta_1$, so we can not conclude about 
				the Borsuk-Ulam property of $[f_1],[f_2]$ in respect to $\mathcal Deck(\tilde {\mathbb T},q_1)=\langle \delta_1,\delta_2\rangle$. Nevertheless we can say that the classes $[f_1],[f_2]$ do not have the Borsuk-Ulam property in respect to the groups $\langle \delta_1\rangle$,  $\langle \delta_2\rangle$, respectively.   
 
		 		\item \label{ex Z2+Z2 reg}
		 		When we change the braids on which the generators $a$ and $b$ are mapped to, we obtain another result.
		 		Consider the map $\Phi_2:\mathbb T\rightarrow D_4(\mathbb D)$ from the torus to the disc such that the correspondent homomorphism $(\Phi_2)_\#$ takes the generators $a$ and $b$ of $\pi_1(\mathbb T)$, to $\sigma_1\sigma_3$ and to the 
		 		Garside element $\Delta_4=(\sigma_1\sigma_2\sigma_3)(\sigma_1\sigma_2)(\sigma_1)$ on $B_4$, respectively.
		 		It is straightforward to see that these two braids commute, therefore $\Phi_2$ is in fact a map.
		 		Then $\rho(\sigma_1\sigma_3)=(12)(34)$ and $\rho(\Delta_4)=(14)(23)$ in $\mathcal S_4$.
		 		Consider $ker \theta_2=\{id,(12)(34),(14)(23),(13)(24)\}$, which induces a covering space $q_2:\tilde { \mathbb T}\rightarrow \mathbb T$ and a lift $\hat \Phi_2=\{f_1,f_2,f_3,f_4\}:\tilde {\mathbb T}\rightarrow F_4(\mathbb D)$.
		 		This is a $4$-sheeted covering and we enumerate the deck-transformation $\delta_1,\delta_2,\delta_3\in \mathcal Deck(\tilde{\mathbb T}, q_2)$ in such a way that
		 		$\Gamma_2(\delta_1)=(12)(34)$, $\Gamma_2(\delta_2)=(14)(23)$ and $\Gamma_2(\delta_3)=(13)(24)$.
		 		Hence $\mathcal Deck(\tilde T,q_2)$is isomorphic to $L'=\langle (12)(34), (13)(24)\rangle \simeq \mathbb Z_2\oplus \mathbb Z_2$
		 		
		 		By Proposition \ref{fsigmak=fdeltak} we have that $f_2=f_1\circ \delta_1$, $f_4=f_1\circ \delta_2$ and $f_3=f_1\circ\delta_3$, which all do not have any coincidence points with $f_1$.
		 		Therefore the homotopy class of $f_1$ does not satisfy the Borsuk-Ulam property in respect to  $\mathcal Deck(\tilde {\mathbb T},q_2)$.
		\end{enumerate}
		Observe that in both examples above, the spaces $\tilde{\mathbb T}$’s are the same, and in both cases $L'$ is isomorphic to $\mathbb{Z}_2 \oplus \mathbb{Z}_2$.
		Moreover, there is only one homotopy class of maps $[\tilde{\mathbb T}, \mathbb D]$.
		However, we obtain different conclusions in these two cases: Example \textit{a)} does not allow us to decide whether the Borsuk–Ulam property holds or not, whereas Example \textit{b)} implies that for any map $f: \tilde{\mathbb T} \to \mathbb D$, the homotopy class $[f]$ does not have the Borsuk–Ulam property. Hence the triple $({\mathbb T}, \mathbb{Z}_2 \oplus \mathbb{Z}_2, {\mathbb D})$ does not have the Borsuk–Ulam property.
		To the best of our knowledge, this result is new.
	\end{ex}
	
	\begin{cor}\label{Z2+Z2 nao BU}
		The triple $({\mathbb T}, \mathbb{Z}_2 \oplus \mathbb{Z}_2, {\mathbb D})$ does not have the Borsuk–Ulam property.
	\end{cor}
	
	\begin{rem}
		The example above shows that the group $\mathbb{Z}_2\oplus\mathbb{Z}_2$ can be embedded in $\mathcal S_4$ in two different ways, where the	
		image of the embedding in the case of Example \ref{ex Z2+Z2}.\ref{ex Z2+Z2 reg} is a regular subgroup and in the case of Example \ref{ex Z2+Z2}.\ref{ex Z2+Z2 not reg} is not. 
		There are many more examples of embeddings of groups besides $\mathbb{Z}_2\oplus\mathbb{Z}_2$ with similar behavior, 
		for example consider the embeddings of the group $\mathbb{Z}_3\oplus\mathbb{Z}_3$   
		having as image the subgroups $G_1=\langle (123)(456)(789), (147)(258)(369)\rangle$ and $G_2=\langle (123), (456)\rangle$, as subgroups of  $\mathcal S_9$.
		In the first case the subgroup is regular and in the second case is not.
		For another example, where the group is not abelian,   consider  the finite group  $\mathcal S_3$, and two embeddings  in  $\mathcal S_9$ having
		image the subgroups  $G_1=\langle   (14)(25)(36), (123)(456)(369)\rangle$  and $G_2=\langle  (12),  (123)\rangle$.
		Note that $G_1$ is regular subgroup, while $G_2$ is not.
			
		Finally observe that given any finite group  $G$ of cardinality $n$ the embedding (either right or left) used in the proof of the classical Cayley's theorem is regular.
	\end{rem}

	We saw in Example \ref{ex Z2+Z2} that both transitivity and regularity of the group $L'$ play a big role in how the lift factors behave.
	If we look at a specific lift factor $f_i$, with $i\in \{1,\ldots,n\}$ fixed, we can consider a property weaker than regularity, only asking for the stabilizer group of $i$ to be trivial, i.e. $L_i'=\{id\}$.
	With these considerations and Example \ref{ex Z2+Z2} in mind we state a sort of converse of Proposition \ref{NonBUmult}.

	\begin{theo}\label{transitive, semireg=> phi}
		Let $\phi:X\multimap Y$ be an $n$-valued non-split map with correspondent map $\Phi:X\rightarrow D_n(Y)$, with $\hat \Phi=(f_1,\ldots,f_n):\tilde X\rightarrow F_n(Y)$ and $L'$ as in Notation \ref{notation non-split}.
		Let furthermore $H'\subseteq L'$ be a subgroup,
		then the following holds;
		\begin{enumerate}[a)]
			\item \label{transitive, semireg=> phi: trans}
				if $L'$ acts transitively on $\{1,\ldots,n\}$,
				 then there exists, 
				for every $j\in \{1,\ldots,n\}$ a deck transformation $\delta_j\in \mathcal Deck(\tilde X,q)$ such that $f_j=f_1\circ \delta_j$;
			\item \label{transitive, semireg=> phi: Li=1}
				if, for a subgroup $H'\subseteq L'$ the stabilizer $H'_i$ of an $i\in \{1,\ldots,n\}$ in $H'$ is trivial, then $[f_i]$ does not satisfy the Borsuk-Ulam property in respect to $H=\Gamma^{-1}(H')$.
		\end{enumerate}
	\end{theo}
	\begin{proof}
		For \textit{a)} let $j\in \{1,\ldots,n \}$, then there exists a permutation $\tau_j \in L'$ such that $\tau_j(1)=j$, by the transitivity of $L'$.
		Let $\delta_j=\Gamma^{-1}(\tau_j)\in \mathcal Deck(\tilde X, q)$,
		then by Proposition \ref{fsigmak=fdeltak} it follows that $f_j=f_{\tau_j(1)}=f_1\circ \delta_j$.
		For \textit{b)} let $\tau\in H'$ and $\delta=\Gamma^{-1}(\tau)\in H=\Gamma^{-1}(H')$, then Proposition \ref{fsigmak=fdeltak} implies that $f_i\circ \delta= f_{\tau(i)}$.
		From the fact that $H'_i=\{id\}$, it follows that $\tau(i)\neq i$ and since $f_1,\ldots,f_n$ are pairwise coincidence-free, so are $f_i$ and $f_i\circ \delta$.
		Therefore does not satisfy the Borsuk-Ulam property in respect to $H$.
	\end{proof}
	\begin{cor}\label{reg=>phi}
		Let	$\phi:X\multimap Y$ be an $n$-valued non-split map as in Theorem \ref{transitive, semireg=> phi}
		and let $L'$ act regularly on $F_n(Y)$.
		Then it holds for every $i\in \{1,\ldots,n\}$ that there is $\delta_i\in \mathcal Deck(\tilde X,q)$ such that $f_i=f_1\circ \delta_i$ and the homotopy class $[f_1]$ does not satisfy the Borsuk-Ulam property in respect to $\mathcal Deck(\tilde X, q)$.
	\end{cor}
	%
	Note 
	that, for every $i\in \{1,\ldots, n\}$, it holds by the Orbit-Stabilizer Theorem that $|\mathbb O_i|\cdot|L'_i|=|L'|>n$, where $|\mathbb O_i|\leq n$.
	Hence, in case that $|L'|=|\mathcal Deck(\tilde x,q)|>n$ it follows that $|L_i'|>1$ and consequently is not trivial.
	Thus Theorem \ref{transitive, semireg=> phi} does not help to conclude if the homotopy class $[f_i]$ does or does not satisfy the Borsuk Ulam property in respect to $\mathcal Deck(\tilde X,q)$.		
		
   In the next proposition we give some properties of a group satisfying the conditions of $H'$ of Theorem \ref{transitive, semireg=> phi}\ref{transitive, semireg=> phi: Li=1}.
   We leave the proof as an exercise for the reader. 
	\begin{prop}\label{Deckiscyclic}
		Let $G\subseteq \mathcal S_n$ be a subgroup such that the stabilizer group $G_s=\{id\}$ and let $\tau_0\in G$ be a non-trivial permutation, where $\tau_0=\tau_1\cdots \tau_l$ is the cyclic decomposition of $\tau$, with $m_j$ the length of the cycle $j=1,\ldots,l$ and labeled such that $\tau_1(s)\neq s$.
		Then the following two statements hold;
		\begin{enumerate}[a)]
			\item $m_j|m_1$ for every $j=1,\ldots,n$;
			\item if $m_i=m_j$, for every $i,j=1,\ldots,l$, then $G=\langle \tau_0 \rangle$.
		\end{enumerate}
	\end{prop}

\section{Nielsen number of $n$-valued non-split maps on connected closed manifolds}\label{sec:Nielsen nr of nval non-split}

	There have been numerous efforts to find ways to compute the Nielsen number of $n$-valued non-split maps.
	Two of these, which are relevant to the current paper, are \cite{staecker2021partitions} and \cite{gonccalves2017fixed}.
	In \cite[Corollary 4.7]{staecker2021partitions} a formula is given for the Nielsen number of an $n$-valued map on finite polyhedras in terms of other multivalued maps.
	In fact, it is shown that the Nielsen number of an $n$-valued map $\phi:X\stackrel{n}{\multimap} X$ is $N(\phi)=N(\psi_1)+\cdots +N(\psi_r)$,
	where $\psi_1,\ldots,\psi_r$ are the (possibly non-singular) maps appearing in Proposition \ref{Phi P-splits}.
	In \cite{gonccalves2017fixed} a method is developed to determine the Nielsen number of certain $n$-valued non-split maps on orientable, connected, closed manifolds in terms of the coincidence number of certain single-valued maps.
	Since methods for computing the coincidence number of such maps are already available in the literature, this approach provides a more practical way to obtain explicit values for the Nielsen number. 
	However, the formula for the Nielsen number given in \cite[Theorem 6]{gonccalves2017fixed} applies only to $n$-valued non-split maps with trivial stabilizer group $L'_i=\{id\}$ for any $i\in\{1,\ldots,n\}$ (see Notation \ref{notation action Li} for the definition of $L'_i$).
	Nevertheless, the authors establish several observations concerning more general cases of $n$-valued maps on these spaces, which play an important role in the arguments developed in the present work, where we expand and refine this formula. 
	We obtain a formula for the Nielsen number of every $n$-valued non-split map on connected closed manifolds (not necessarily orientable) in terms of the coincidence number of certain single-valued maps. 
	In particular, this includes $n$-valued non-split maps with non-trivial stabilizer groups. 
	Our approach builds on the framework of \cite{gonccalves2017fixed} together with the properties established in Section \ref{sec:lift factors}.

	For this purpose let $X$ be a closed connected manifold and let $\phi:X\stackrel{n}{\multimap} X$ be an $n$-valued non-split map with correspondent map $\Phi:X\rightarrow D_n(X)$.
	Let furthermore $q:\tilde X\rightarrow X$, $\hat \Phi=(f_1,\ldots,f_n):\tilde X\rightarrow X$, $L'\subseteq \mathcal S_n$, $\mathcal Deck(\tilde X,q)$, $\Gamma:\mathcal Deck(\tilde X,q)\rightarrow L'$, $L'_i$ and $\mathbb O_i$ etc. be as in Notations \ref{notation non-split} and \ref{notation action Li}.
	Consider $x_0\in X$, a fixed point of $\Phi$, and let $\tilde x_1\in q^{-1}(x_0)$ be a point of the fiber.
	Then there is an $i\in \{1,\ldots,n\}$ such that $f_i(\tilde x_1)=x_0$ and consequently $\tilde x_1\in Coin(q,f_i)$.
	Conversely, if $\tilde x_2\in Coin(q,f_i)$, for an $\{1,\ldots,n\}$, then $f_i(\tilde x_2)=q(\tilde x_2)$ is contained in $\Phi(q(\tilde x_2))$ and consequently  $q(\tilde x_2)$ is a fixed point of $\Phi$, i.e. $q(\tilde x_2)\in Fix(\Phi)$.
	 
	Hence $q:\tilde X\rightarrow X$ induces a surjection between $\bigcup_{i=1}^n Coin(q,f_i)$ and $Fix(\phi)$, by taking $\tilde x_1\in \bigcup_{i=1}^n Coin(q,f_i)$ to $q(\tilde x_1)$.
	To see that this map is not injective we look more into which points of the fiber of a fixed point are coincidence points of which lift factors.
	The next lemma, proven first in \cite{gonccalves2017fixed} and rewritten using the notation established in the current work, explores the connection between coincidence points on the fiber of a fixed point of $\phi$ and the group action by the group of deck-transformation.

		\begin{lem}\cite[Lemma 16]{gonccalves2017fixed}\label{Coini; Lij<=> Coinj}
			Let $\phi:X\stackrel{n}{\multimap}X$ be an n-valued non-split map inducing the concepts in Notations \ref{notation non-split} and \ref{notation action Li}.
			Also, let $x_0\in Fix(\phi)$, $\tilde x_1,\tilde x_2\in q^{-1}(x_0)\subset\tilde X$, $i,j\in \{1,\ldots,n\}$ and $\tilde x_1\in Coin(q,f_i)$.
			Let furthermore $\delta\in \mathcal Deck(\tilde X,q)$ be the deck-transformation such that $\delta(\tilde x_2)=\tilde x_1$ and $\tau=\Gamma(\delta)$.
			Then $\tilde x_2\in Coin(q,f_j)$, if and only if, $\tau(i)=j$.
		\end{lem}
		\begin{proof}
			By hypothesis and Proposition \ref{fsigmak=fdeltak} it holds that $f_{\tau(i)}(\tilde x_2)=f_i(\delta(\tilde x_2))=f_i(\tilde x_1)=q(\tilde x_1)=x_0$.
			Note furthermore that, since $f_{\tau(i)}$ and $f_j$ are lift factors, they are either coincidence-free or $\tau(i)$ is equal to $j$.
			Therefore $f_j(\tilde x_2)=x_0$, if only if, $\tau(i)=j$.
		\end{proof}
	
		By this lemma, for every coincidence point $\tilde x_1\in  Coin(q,f_i)$, every point of the orbit of $\tilde x_1$ by $\mathcal Deck(\tilde X, q)$ is a coincidence point of q and a certain lift factor. 
		Therefore $q$ takes the whole orbit of $\tilde x_1$ to the same point, hence the map induced by $q$ between $\bigcup_{i=1}^n Coin(q,f_i)$ and $Fix(\phi)$ is not injective.
		To be able to compare the Nielsen fixed point number of $\phi$ to the coincidence number of certain lift factors and $q$ it would be nice to find a bijection between fixed points and coincidence points.
		One way of "transforming" this map to a bijection is to "stop to consider" some of these coincidence points.
		Considering the group action $L'$ instead of $\mathcal Deck(\tilde X,q)$, this lemma implies that if $i\in \{1,\ldots,n\}$ and $j\in \mathbb O_i$ is in the orbit of $i$ under $L'$, then $Coin(q,f_i)=Coin(q,f_j)$.
		Hence we only have to consider the coincidence points of lift factors with indices in different orbits.
		For simplicity sake we take the smallest number of each orbit, i.e. $\{f_i: i\in \mathbb I_0\}$.
		Then $q$ induces a bijection from $\bigcup_{i\in \mathbb I_0} Coin(q,f_i)$ to $Fix(\phi)$, in case that there are only trivial stabilizer groups, i.e. $L_i=\{id\}$ for every $i\in \mathbb I_0$.
		
		That leaves us with the case where, for an $i\in \mathbb I_0$, there is a non-trivial stabilizer group $L'_i$.
		In this case, there is a non-trivial permutation $\tau \in L_i$, i.e. $\tau(i)=i$ and $\tau\neq id$.
		Let $\delta=\Gamma^{-1}(\tau)$, $\tilde x_1\in Coin(q,f_i)$ and set $\tilde x_2=\delta(\tilde x_1)$.
		Then it follows, by Lemma \ref{Coini; Lij<=> Coinj}, that $\tilde x_2$ is also a coincidence point of $f_i$ and $q$, i.e. $\tilde x_2\in Coin(q,f_i)$.
		Since both $\tilde x_1$ and $\tilde x_2$ are mapped by $q$ to $ x_0\in Fix(\phi)$, it follows that the map from $\bigcup_{i\in \mathbb I_0} Coin(q,f_i)$ to $Fix(\phi)$ induced by $q$ is not injective.
		This can be avoided if we instead consider the quotient space $\tilde X/L_i $ and the maps $q_i,\bar f_i:X/L_i \rightarrow X$, which are induced by $q$ and $f_i$ respectively (see Notation \ref{notation action Li}).
		Then the projection $p_i:\tilde X\rightarrow \tilde X/L_i$ takes $\tilde x_1$ to its equivalence class $[\tilde x_1]\in \tilde X/L_i$.
		Note that by Lemma \ref{Coini; Lij<=> Coinj} the projection $p_i$ takes all points in $q^{-1}(x_0)\cap Coin(q,f_i)$ to $[\tilde x_1]$.
		That is, the points in the fiber of $x_0$, which are contained in the pre-image $f_i^{-1}(x_0)$ are mapped by $p_i$ to $[\tilde x_1]$, in particular $p_i(\tilde x_2)=[\tilde x_1]$.
		\begin{equation*}
			\begin{tikzcd}[column sep=huge]
				\tilde x_1 \arrow[d,mapsto]&[-11ex] \in&[-12ex]\tilde X\arrow[d,"p_i"]\arrow[dd,bend left=60, "q"]\\
				\left[\tilde x_1\right]\arrow[d,mapsto]&[-11ex] \in &[-12ex]\tilde X/L_i \arrow[d,"q_i"]\\
				x_0&[-11ex] \in &[-12ex]X	
			\end{tikzcd}
		\end{equation*}			
		Using these observations we define a map $q_{Coin}: \bigcup_{i\in \mathbb I_0} Coin(q_i,\bar f_i)\rightarrow Fix(\phi)$, taking an equivalence class $[\tilde x_1]\in Coin(q_i,\bar f_i)\subset \tilde X/L_i$ to $q_i([\tilde x_1])=q(\tilde x_1)$.
		In the next proposition we show that $q_{Coin}$ is well-defined and bijective.			
	\color{black}

	\begin{prop}\label{qCoin is a bijection}
		Let $\phi:X\stackrel{n}{\multimap} X$ be an $n$-valued non-split map with correspondent map $\Phi:X\rightarrow D_n({X})$ inducing the structure summarized in Notations \ref{notation non-split} and \ref{notation action Li}.
		Then the map $q_{Coin}:\bigcup_{i\in \mathbb I_0} Coin(q_i,\bar f_i)\rightarrow Fix(\phi)$, taking an equivalence class $[\tilde x_1]\in Coin(q_i,\bar f_i)$ to $q_i([\tilde x_1])=q(\tilde x_1)$, is a bijection.
	\end{prop}

		\begin{proof}
			To show that $q_{Coin}$ is well-defined let $[\tilde x_1]\in  \bigcup_{i\in \mathbb I_0} Coin(q_i,\bar f_i)$ and let $\tilde x_2\in [\tilde x_1]$.
			Then there exists an $i\in \mathbb I_0$, such that $[\tilde x_1]\in Coin(q_i,\bar f_i)$, and a deck-transformation $\delta\in L_i$ such that $\delta(\tilde x_1)=\tilde x_2$.
			Therefore $q_{Coin}([\tilde x_1])=q_i([\tilde x_1])=q(\tilde x_1)=q(\tilde x_2)=q_i([\tilde x_2])=q_{Coin}([\tilde x_2])$.
			It also holds that $q_{Coin}$ associates every equivalent class $[\tilde x_1]\in Coin(q_i,\bar f_i)$, for any $i\in \mathbb I_0$, to a fixed point of $\phi$.
			In fact $q_{Coin}([\tilde x_1])=q_i([\tilde x_1])=q(\tilde x_1)$, which is a fixed point of $\phi$, since $q(\tilde x_1)=\bar f_i([\tilde x_1])=f_i(\tilde x_1)\in \phi(q(\tilde x_1)) $.
			Hence $q_{Coin}$ is well-defined.
		
			To show that $q_{Coin}$ is surjective let $x_0\in Fix (\phi)$, i.e. $x_0\in \phi(x_0)$, and take a point in it's fiber $\tilde x_1\in q^{-1}(x_0)$.
			Then, since $\hat \Phi=\{f_1,\ldots,f_n\}$ is the lift of $\Phi\circ q$, there exists $i\in \{1,\ldots,n\}$ such that $f_i(\tilde x_1)=x_0$.
			Hence $q_i(p_i(\tilde x_1))=q(\tilde x_1)=x_0$ and $\bar f_i(p_i(\tilde x_1))=f_i(\tilde x_1)=x_0$ and consequently $[\tilde x_1]=p_i(\tilde x_1)\in Coin(q_i, \bar f_i)$ and $q_{Coin}$ is surjective.
			
			To show that $q_{Coin}$ is injective suppose that $x_0\in Fix(\phi)$ and $[\tilde x_1],[\tilde x_2] \in \bigcup_{i\in \mathbb I_0} Coin(q_i,\bar f_i)$ such that
			$q_{Coin}([\tilde x_1])=q_{Coin}([\tilde x_2])=x_0$
			Then there are $j,k\in \mathbb I_0$ such that $[\tilde x_1]\in Coin( q_{j},\bar f_{j})$ and $[\tilde x_2]\in Coin(q_{k},\bar f_{k})$ and consequently
			$\tilde x_1\in Coin(q,f_{j})$ and $\tilde x_2\in Coin(q, f_{k})$.
			Since $\tilde x_1$ and $\tilde x_2$ are in the same fiber,
			there exists a deck-transformation $\delta\in \mathcal Deck(\tilde X, q)$ such that $\delta(\tilde x_2)=\tilde x_1$.
			By Lemma \ref{Coini; Lij<=> Coinj} it holds that $\Gamma(\delta)$ takes $j$ to $k$.
			By definition of $\mathbb I_0$ it contains only one element of each orbit.
			Hence $j=k$ and $\delta\in L_j$, consequently $[\tilde x_1]=[\tilde x_2]$ in $\tilde X/L_j$.
	\end{proof}
	
	\begin{prop}\label{coincidence class bij fixed point class}
		Let $i\in \mathbb I_0$ and $[\tilde x_1]\in Coin(q_i,\bar f_i)$, then $q_{Coin}$ maps the Nielsen coincidence class of $[\tilde x_1]$ bijectively to the Nielsen fixed point class of $x_0=q_i([\tilde x_1])$.
	\end{prop}
	\begin{proof}		
		Let $C$ be a fixed point class of $\phi$, let $x_0\in C$ and $[\tilde x_1]=q_{Coin}^{-1}(x_0)$, which is contained in $Coin(q_i, \bar f_i)$, for some $i\in \mathbb I_0$.
		Let $\tilde C$ be the Nielsen coincidence class of $\tilde x_1$ of the pair $(q,f_i)$,
		then $q$ maps $\tilde C$ surjectively to $C$ by \cite[Lemma 18]{gonccalves2017fixed}.
		It holds that $p_i(\tilde C)$ is a Nielsen coincidence class of the pair $(q_i, \bar f_i)$ and $q_i$ maps $p_i(\tilde C)$ surjectively to $C$.
		Since $q_{Coin}$ is, by Proposition \ref{qCoin is a bijection}, a bijection.
	\end{proof}

		Note that, for any $i\in \{1,\ldots,n\}$, it holds that the covering space maps $q:\tilde X\rightarrow X$, $p_i:\tilde X\rightarrow \tilde X/L_i$ and $q_i:\tilde X/L_i\rightarrow X$ are orientation-true (see for example \cite{gonccalves2002roots}).
		Recall that for two connected manifolds $\tilde X, X$, a map $g:\tilde X\rightarrow X$ is {\it orientation-true}, if for any loop $\gamma:I\rightarrow \tilde X$ the loop $g\circ \gamma:I\rightarrow X$ preserves (reverses) orientation, if and only if, $\gamma$ preserves (reverses) orientation (see \cite[page 406]{brown2005handbook}).
		Consider two maps $f, g: \tilde X\rightarrow X$ with $g$ orientation-true, then the coincidence index for the pair $(g,f)$ of an isolated coincidence point is an integer defined as the local index of the pair $(g,f)$, (for more details see \cite[Definition 5.1.]{gonccalves1997lefschetz}).
		The map $g$ being orientation-true guarantees that there is a consistent global choice of local orientations.
	\color{black}

	\begin{prop}\label{index fix=coin}
		Let $X$ be a closed connected (not necessarily orientable) manifold, $\phi : X \multimap X$ be an $n$-valued non-split map with correspondent map $\Phi : X \to D_n(X)$ and $\hat{\Phi} = \{f_1, \ldots, f_n\} : \tilde{X} \to F_n(X)$ be a lift of $q \circ \Phi$ (see Notation \ref{notation non-split}).
		Let furthermore $x_0 \in Fix(\phi)$ be an isolated fixed point of $\phi$, let $[\tilde{x}_1]=q_{Coin}^{-1}(x_0)$ and let $i \in \mathbb I_0$ be such that $[\tilde{x}_1]\in Coin(q_i, \bar f_i)$.
		Then the fixed point index of $\phi$ at $x_0$ equals the coincidence index of the pair $(q_i, \bar f_j)$ at $[\tilde{x}_1]$.
		Moreover the Nielsen fixed point class of $x_0$ is essential, if and only if, the Nielsen coincidence class of $[\tilde{x}_1]$ of the pair $(q_i, \bar f_i)$ is essential.
	\end{prop}
	\begin{proof}
		Since $[\tilde{x}_1]=q_{Coin}^{-1}(x_0)$, it holds, by definition of $q_{Coin}$ and $q_i$, that $x_0=q_{Coin}([\tilde x_1])=q_i([\tilde x_1])=q(\tilde x_1)$.
		As $x_0$ is an isolated fixed point and $X$ is a manifold, there exist contractible neighborhoods $U\subset X$ of $x_0$ and $\tilde U\subseteq \tilde X$ of $\tilde x_1$, such that $Fix(\phi)\cap U=\{x_0\}$, $q^{-1}(x_0)\cap \tilde U=\{\tilde x_1\}$ and $ q_{|\tilde U}:\tilde U\rightarrow U$ is a homeomorphism.
		It holds that $U$ satisfies the conditions of the Splitting Lemma (see \cite[Lemma 1]{schirmer1984fix}) and therefore there are maps $h_1,\ldots,h_n:U\rightarrow X$ such that $\Phi_{|U}=\{h_1,\ldots,h_n\}:U\rightarrow D_n(X)$.
		Consequently $\{h_1,\ldots,h_n\}=\{f_1,\ldots,f_n\}\circ q_{|\tilde U}^{-1}$. 
		Without loss of generality, suppose we labeled $h_1,\ldots, h_n$ such that $h_i=f_i\circ q_{|\tilde U}^{-1}$.
		It holds that $f_i=\bar f_i \circ p_i$ and $q=q_i\circ p_i$
		(see commuting Diagram \eqref{bar qi & bar fi}), therefore setting $\tilde V=p_i(\tilde U)$ we obtain that $h_i=\bar f_i\circ (q_i)_{|\tilde V}^{-1}$.
		Since $[\tilde{x}_1 ]\in Coin(q_i, \bar f_i)$ it follows that $x_0$ is a fixed point of $h_i$.
		By definition of the fixed point index it holds
		\begin{equation*}
			ind(\phi, U)=ind(h_i,U)=ind(\bar f_i\circ ((q_i)_{|\tilde V})^{-1}, U),
		\end{equation*}
		which is exactly the local coincidence index of the pair $(q_i,\bar f_i)$ in $\tilde V$, where the local orientation in $\tilde V$ is determined by the local homeomorphism $(q_i)_{|\tilde V}$.
		Since $q_i: \tilde X/L_i\rightarrow X$ is a covering map it is orientation-true, thus the local index extends globally to $\tilde X/L_i$ (see \cite[Definition 5.1.]{gonccalves1997lefschetz}).
		The second statement follows from the first.
	\end{proof}

	We now proceed to the proof of the primary result in this section, for a recollection of the notation and different concept used, see Notation \ref{notation non-split} and \ref{notation action Li}.
	
	\begin{theo}\label{Nielsennumber non-orientable}
		Let $X$ be a connected compact manifold without boundary (orientable or non-orientable), and let $\phi:X\stackrel{n}{\multimap} X$ be an $n$-valued non-split map with correspondent map $\Phi:X\rightarrow D_n(X)$.
		Let $q:\tilde X\rightarrow X$ be the covering space induced by $\Phi^{-1}(P_n(X))$, let $\hat{\Phi}=\{f_1,\ldots,f_n\}:\tilde{X}\rightarrow F_n(X)$ be a lift of $\Phi\circ q$.
		Let $L'$ be as in Notation \ref{notation non-split} and $\mathbb I_0$, $L_{i}'$, $ L_i$ and $q_i, \bar f_i:\tilde X/L_i\rightarrow X$ as in Notation \ref{notation action Li}.
		Then the Nielsen number of $\phi$ is equal to
		\begin{equation*}
			N(\phi) = \sum_{i\in  \mathbb I_0}N(q_i,\bar f_{i}).
		\end{equation*}
	\end{theo}
	\begin{proof}
		By Proposition \ref{coincidence class bij fixed point class}  $q_{Coin}: \bigcup_{i\in \mathbb I_0} Coin(q_i,\bar f_i)   \rightarrow Fix(\phi)$ is a bijection, that takes, by Proposition \ref{qCoin is a bijection} every Nielsen coincidence class bijectively to a fixed point class.
		Proposition \ref{index fix=coin} implies that a coincidence class $[\tilde x_1]$ of a pair $(q_i, \bar f_i)$ is essential, if and only if, the correspondent fixed point class $q_i([\tilde x_1])$ is essential.	
	\end{proof}

	\begin{cor}
		Let $X$ be a compact manifold without boundary and $\phi:X\stackrel{2}{\multimap} X$ a non-split $2$-valued map. Then $N(\phi)=N(q,f_1)=N(q,f_2)$.
	\end{cor}
	\begin{proof}
		Since $\phi$ is non-split, $L \simeq L' \simeq \mathbb{Z}_2$, there is a single orbit $\mathbb{O}_i = \{1,2\}$, and $L_1=L_2 = \{id\}$.  
		Then by Theorem \ref{Nielsennumber non-orientable}, $N(\phi) = N(q,f_j)$, for $j \in \{1,2\}$.
	\end{proof}
	

\bibliographystyle{plain} 

\begin{thebibliography}{10}

	
	\bibitem{brown2005handbook}
	Robert~F Brown, Massimo Furi, Lech G{\'o}rniewicz, and Boju Jiang.
	\newblock {\em Handbook of topological fixed point theory}.
	\newblock Springer, 2005.
	
	\bibitem{brown2018topology}
	Robert~F Brown and Daciberg~Lima Goncalves.
	\newblock On the topology of $n$-valued maps.
	\newblock {\em Advances in Fixed Point Theory}, 8(2):205--220, 2018.
	
	\bibitem{brown2023lift}
	Robert~F. Brown and Daciberg~L. Gon{\c{c}}alves.
	\newblock Lift factors for the Nielsen root theory of $n$-valued maps.
	\newblock {\em Topological Methods in Nonlinear Analysis}, 61(1):269--289,
	2023.
	
	\bibitem{gonccalves1997lefschetz}
	Daciberg~L. Gon{\c{c}}alves and Jerzy Jezierski.
	\newblock Lefschetz coincidence formula on non-orientable manifolds.
	\newblock {\em Fundamenta Mathematicae}, 153(1):1--23, 1997.

	\bibitem{gonccalves2017fixed}
	Daciberg~L. Gon{\c{c}}alves and John Guaschi.
	\newblock Fixed points of $n$-valued maps on surfaces and the Wecken property—a
	configuration space approach.
	\newblock {\em Science China Mathematics}, 60(9):1561--1574, 2017.
	
	\bibitem{gonccalves2018fixed}
	Daciberg~L. Gon{\c{c}}alves and John Guaschi.
	\newblock Fixed points of $n$-valued maps, the fixed point property and the case
	of surfaces—a braid approach.
	\newblock {\em Indagationes Mathematicae}, 29(1):91--124, 2018.
	
	\bibitem{gonccalves2002roots}
	Daciberg~L. Gon{\c{c}}alves, Elena Kudryavtseva and Heiner Zieschang.
	\newblock Roots of mappings on nonorientable surfaces and equations in free groups.
	\newblock {\em manuscripta mathematica}, 107(3): 311--341, 2002.	

	\bibitem{gonccalves2024borsuk}
	Daciberg~L. Gon{\c{c}}alves and Jes{\'u}s Gonz{\'a}lez.
	\newblock Borsuk--Ulam property for graphs II: The $Z_n$-action.
	\newblock {\em Lobachevskii Journal of Mathematics}, 46(3):1057--1075, 2025.
	
	\bibitem{gonccalves2022free}
	Daciberg~L. Gon{\c{c}}alves, John Guaschi and Vinicius~Casteluber Laass.
	\newblock Free cyclic actions on surfaces and the Borsuk-Ulam theorem.
	\newblock {\em Acta Mathematica Sinica, English Series}, 38(10):1803--1822,
	2022.
	
	\bibitem{gonccalves2025splitting}
	Daciberg~L. Gon{\c{c}}alves, John Guaschi and Carolina de Miranda e Pereiro.
	\newblock The splitting of generalisations of the Fadell-Neuwirth short exact sequence.
	\newblock {\em J. Algebra} 694: 629--675, 2026.
	
	\bibitem{hatcher2002algebraic}
	Allen Hatcher.
	\newblock {\em Algebraic Topology}.
	\newblock Cambridge University Press, 2002.
	
	\bibitem{maues2025computation}
	Bartira Mau{\'e}s.
	\newblock Computation of the Nielsen fixed point number for $2$-valued non-split
	maps on the Klein bottle.
	\newblock {\em arXiv preprint arXiv:2504.20171}, 2025.
	
	\bibitem{schirmer1984fix}
	Helga Schirmer.
	\newblock Fix-finite approximation of $n$-valued multifunctions.
	\newblock {\em Fundamenta Mathematicae}, 1(121):73--80, 1984.
	
	\bibitem{schirmer1984index}
	Helga Schirmer.
	\newblock An index and a Nielsen number for $n$-valued multifunctions.
	\newblock {\em Fundamenta Mathematicae}, 124(3):207--219, 1984.
	
	\bibitem{schirmer1985minimum}
	Helga Schirmer.
	\newblock A minimum theorem for $n$-valued multifunctions.
	\newblock {\em Fundamenta Mathematicae}, 126(1):83--92, 1985.
	
	\bibitem{staecker2021partitions}
	P. Christopher Staecker,
	\newblock Partitions of $ n $-valued maps.
	\newblock {\em arXiv preprint arXiv:2101.09326}, 2021. 
	
\end{thebibliography}

\end{document}